\documentclass[12pt]{amsart} 
\usepackage{amssymb} 
\usepackage{amsfonts} 
\usepackage{amsmath} 
\usepackage{amsthm} 
\usepackage{array} 
\usepackage{geometry} 
\usepackage{graphicx} 
\usepackage{mathrsfs} 
\usepackage{pifont} 
\usepackage[all,ps,cmtip]{xy} 
\usepackage{verbatim} 
\usepackage{stmaryrd}

\theoremstyle{plain} 
\newtheorem{thm}{Theorem}[subsection] 
\newtheorem{cor}[thm]{Corollary} 
\newtheorem{lem}[thm]{Lemma} 
\newtheorem{prop}[thm]{Proposition} 

\theoremstyle{definition} 
\newtheorem{defn}[thm]{Definition} 
\newtheorem{pr}[thm]{Problem} 
\newtheorem*{notation}{Notation} 
\newtheorem{rem}[thm]{Remark} 
\newtheorem*{pfoflem}{Proof of the lemma} 
\theoremstyle{remark} 

\DeclareMathOperator{\Spec}{Spec}
\DeclareMathOperator{\Ker}{Ker}
\DeclareMathOperator{\Coker}{Coker}

\numberwithin{equation}{subsection}
\newcommand{\N}{\mathbb{N}} 
\newcommand{\Z}{\mathbb{Z}}     
\newcommand{\C}{\mathbb{C}} 
\newcommand{\A}{\mathbb{A}}     
\newcommand{\G}{\mathbb{G}} 
\newcommand{\Proj}{\mathbb{P}}    
\newcommand{\DD}{\mathbb{D}}    
\newcommand{\E}{\textbf{E}} 
\newcommand{\F}{\textbf{F}}  
\newcommand{\BL}{\textbf{L}} 
\newcommand{\Bev}{\textbf{ev}} 
\newcommand{\Br}{\textbf{r}}
\newcommand{\Bc}{\textbf{c}} 
\newcommand{\sA}{\mathcal{A}}
\newcommand{\sB}{\mathcal{B}}
\newcommand{\sC}{\mathcal{C}}
\newcommand{\sD}{\mathcal{D}}

\newcommand{\sG}{\mathcal{G}}

\newcommand{\sK}{\mathcal{K}}
\newcommand{\sL}{\mathcal{L}} 
\newcommand{\sM}{\mathcal{M}}
\newcommand{\sN}{\mathcal{N}}
\newcommand{\sO}{\mathcal{O}}

\newcommand{\sT}{\mathcal{T}}
\newcommand{\sW}{\mathcal{W}} 
\newcommand{\sX}{\mathcal{X}}
\newcommand{\sU}{\mathcal{U}}

\newcommand{\fM}{\mathfrak{M}} 

\newcommand{\fQ}{\mathfrak{Q}}
\newcommand{\fT}{\mathfrak{T}}

\newcommand{\CharM}{\overline{\mathcal{M}}}

\begin{document}
\title{The degeneration formula for logarithmic expanded degenerations}


\author{Qile Chen}
\address{Qile Chen, Department of Mathematics, Box 1917, Brown University,
Providence, RI, 02912, U.S.A} \email{q.chen@math.brown.com}

\maketitle
\setcounter{tocdepth}{1}
\tableofcontents
%
\section{Introduction}

Throughout this paper, we work over $\C$, the field of complex numbers.

\subsection{Degeneration formula using expanded degenerations}
Gromov-Witten theory has been established and intensively studied in the past decades for compact symplectic manifolds under the symplectic setting, and for smooth projective complex varieties under the algebraic setting. We refer to \cite{C-K} for an extensive bibliography. This theory was later extended to the case of smooth Deligne-Mumford stacks, see \cite{CR,AGV}.

An important remaining problem is how to calculate Gromov-Witten invariants in general. The method we are interested in here is by means of degenerations. Consider $\pi: W \to B$, a flat, projective family of schemes over a smooth, connected, and possibly non-proper curve $B$. Let $0\in B$ be a closed point such that $\pi$ is smooth away from $W_{0}=W \times_{B}0$, and the central fiber $W_{0}$ is reducible with two smooth components $X_{1}$ and $X_{2}$ intersecting transversally along a smooth divisor $D\subset W_{0}$. We view $D$ as a smooth divisor in $X_{i}$, and write $D_{i}\subset X_{i}$. Then we have two smooth pairs $(X_{1},D_{1})$ and $(X_{2},D_{2})$. It is well-known that the smooth fibers all have the same Gromov-Witten invariants. It is natural to ask the following questions:

\begin{pr}\label{pr:central}
Can we define Gromov-Witten theories for the singular fiber $W_{0}$ and the pairs $(X_{i},D_{i})$ such that 
\begin{enumerate}
 \item The Gromov-Witten invariants of $W_{0}$ are the same as those of the smooth fibers;
 \item There exists a degeneration formula that relates the Gromov-Witten invariants of $W_{0}$, and hence of the smooth fibers, to the relative Gromov-Witten invariants of $(X_{i},D_{i})$?
\end{enumerate}
\end{pr}

In the situation described above, these problems were answered under symplectic setting by A.M. Li and Y. Ruan \cite{LR}, and about the same time by E.N. Ionel and T. Parker \cite{IP1,IP2}. On the algebraic side, these were worked out by Jun Li \cite{Jun1,Jun2}. Their approach uses the method of expanded degenerations --- a surgery on the target which forces the stable maps to be non-degenerate, namely no components of the curves mapping into the singular locus of $W_{0}$ or the divisor $D_{i}$. 

The objects studied in Jun Li's setting are called the predeformable maps. One difficulty of this theory is that the predeformability is in general not an open condition. The usual deformation theory of stable maps, hence the usual construction of perfect obstruction theory, does not work for predeformable maps. Jun Li's study of the deformation theory of predeformable maps was inspired by log structures. However he did not use log structures explicitly, as at the time he developed the theory of predeformable maps, the theory of logarithmic cotangent complex \cite{LogCot} has not been developed yet. 

Another approach that combines the method of expanded degenerations and the orbifold techniques was recently introduced by Dan Abramovich and Barbara Fantechi \cite{AF}. By taking the suitable root stacks along singular locus, they constructed transversal maps over each predeformable map. Since the transversality is open, the construction of perfect obstruction theory for transversal maps is more transparent. A degeneration formula was proved in \cite{AF} by systematically using the orbifold techniques.

Based on Jun Li's construction, B. Kim introduced his notion of logarithmic stable maps \cite{Kim} by putting certain log structures along the nodes of the source curves, and the singular locus of the targets. This can be viewed as a generalization of the idea of admissible covers revisited by Mochizuki \cite{Mochizuki} using log structures. Then using logarithmic cotangent complex in the sense of \cite{LogCot}, B. Kim construct a perfect obstruction theory for the stack of the log stable maps. However, he did not give a degeneration formula under this setting.

\subsection{The goal and outline of this paper}
This paper is aimed at obtaining the degeneration formula in Theorem \ref{thm:main} by applying the method developed in \cite{AF} to Kim's log stable maps. However for the purpose of the degeneration formula, the log structures for log stable maps in this paper are slightly different from the one in \cite{Kim}: besides Kim's log structure along singular locus, the log maps in this paper were also equipped with the standard log structures coming from both marked points of the source curves and the smooth divisors of the targets. The stack parameterizing the log stable maps in this paper will be constructed by following the same proof in \cite{Kim}. We refer to Section \ref{sec:source} for the log structures we used on curves, and Section \ref{sec:general-target} for the definition of log stable maps, and the construction of the stacks. 

The idea of constructing virtual fundamental class with log cotangent complex in the sense of \cite{LogCot} was first introduced in \cite{Kim}. In this paper, we will adopt this idea. However, the formation of our virtual fundamental class will be similar to the one in \cite{AF}. Section \ref{sec:fundamental-class} is devoted to construct the virtual fundamental class, and study its behavior under the base change, which will be important in the proof of the degeneration formula.

In Section \ref{sec:target}, we will introduce the targets and their moduli spaces that are related to Problem \ref{pr:central}. It will be proved that those stacks are similar or even identical to the ones in \cite{AF}. Then we collect some splitting results in \cite{AF} for those stacks. The Gromov-Witten invariants for the targets introduced in Section \ref{sec:target} will be defined in Section \ref{sec:GW}.

Section \ref{sec:degeneration-formula} is devoted to prove the degeneration formula using the method in \cite{AF}. The main difficulty here is to study the gluing of the log maps. Unlike the situation in \cite{AF}, the gluing of log maps can not be expressed as a push-out diagram. However, the possible log structures on the glued underlying maps can be expressed as tuples of isomorphisms of certain line bundles. This still gives a way for us to compare the virtual fundamental classes of the stacks of relative log stable maps and the stack of log stable maps with fixed splittings.

Finally we give an appendix, which collect some results of log geometry that we will use in this paper. 

We would like to point out that the degeneration formula constructed in this paper has the same formation as in \cite{AF}. Indeed, we further expect the theory of transversal maps in \cite{AF} is equivalent to the theory of log stable maps in the sense of this paper, and are planning to explore this in the subsequent paper. 

\subsection{Other approaches}
As early as in 2001, another approach using logarithmic structures without expansions was first proposed by Bernd Siebert \cite{Siebert}. The goal here
is also to obtain the degeneration formula, but in a much more general situation, such as normal crossing divisors. However, the program has been on hold for a while, since Mark Gross and Bernd Siebert were working on other projects in mirror symmetry. Only recently they have taken up the unfinished project of Siebert jointly. In particular, they succeeded to find a definition of basic log maps, a crucial ingredient for a good moduli theory of stable log maps with a fixed target \cite{GS2}. Their definition builds on insights from tropical geometry, obtained by probing the stack of
log maps using the standard log point. This theory is expected to cover the case of targets with arbitrary fine and saturated log structures, or even with relatively coherent log structures. In particular, this includes the class of deformations Gross and Siebert need in their mirror symmetry program \cite[2.2]{GS1}.

Along this approach, another theory of minimal log stable maps was established recently by Dan Abramovich and the author \cite{Chen, AC}, which gives a compactifications of moduli spaces of stable maps relative to certain toric divisors. This covers many cases of interesting, such as a variety with a simple normal crossings divisor, or a simple normal crossings degeneration of a variety with simple normal crossings singularities. Using the theory of minimal log stable maps, a further program for the degeneration formulas in more general situations is on its way.

A different approach using exploded manifolds to studying holomorphic curves was recently introduced by Brett Parker in \cite{Parker1}, \cite{Parker3}, and \cite{Parker2}. It also aimed at defining and computing relative and degenerated Gromov-Witten theories in general situation. It was pointed out by Mark Gross that this approach is parallel, and possibly equivalent to the logarithmic approach.

\subsection{Acknowledgements}
I would like to thank my advisor Dan Abramovich for suggesting the problem of this paper, giving me many enlightening suggestions and encouragement. I would also like to thank Mark Gross, and Bernd Siebert for their helpful conversations.

\section{Minimal Log Prestable curves}\label{sec:source}

\subsection{Log prestable curves}\label{ss:LogPresCurve}
Let $(C\to S, \{\Sigma_{i}\}_{i=1}^{n})$ be a usual genus $g$, $n$-marked prestable curve over $S$, where $\Sigma_{i}$ is the $i$-th marking for $i=0,\cdots,n$. 

First consider the family of curves $C\to S$ without markings. By \cite{FKato} and \cite{LogSS}, there is a pair of canonical log structure $\sM_{C}^{C/S}$ and $\sM_{S}^{C/S}$ on the fiber and base respectively, and a log smooth, integral morphism $\pi:(C,\sM_{C}^{C/S})\to (S,\sM_{S}^{C/S})$ whose underlying map coincides with $C\to S$. This canonical log structure has the following universal property that for any other log smooth, integral morphism $\pi:(C,\sM_{C})\to (S,\sM_{S})$, there is a unique (up to a unique isomorphism) pair of morphisms of fine log structures $\sM_{S}^{C/S}\to \sM_{S}$ and $\sM_{C}^{C/S}\to \sM_{C}$ on $S$ and $C$ respectively, which give the following catesian diagram:
\begin{equation}\label{diag:curve-univ}
\xymatrix{
(C,\sM_{C}) \ar[r] \ar[d] & (C,\sM_{C}^{C/S}) \ar[d]\\
(S,\sM_{S}) \ar[r] & (C,\sM_{S}^{C/S}).
}
\end{equation}

\begin{rem}\label{rem:DecompLogCurve}
Let $Sing\{C/S\}$ be the set of connected singular locus of $C$ over $S$, and assume that $S$ is a geometric point. For each $p\in Sing\{C/S\}$, there is a rank $1$, locally free log structure $\sN_{p}$ on $S$, which smooths the singular locus $p$. We have the decomposition:
\[\sM_{S}^{C/S}\cong \sum_{p\in Sing\{C/S\}}{\sN_{p}},\]
where the sum is taken over $\sO_{S}^{*}$.
\end{rem}

In order to do the degeneration formula, we would like to also put log structure on markings. Note that the $n$ markings $\{\Sigma_{i}\}$ corresponds to $n$ disjoint smooth divisors in $C$ over $S$. By \cite[1.5(1)]{KKato}, there is a canonical log structure associated to the smooth divisors described as follows. \'Etale locally we choose a generator $\sigma_{i}$ for the ideal sheaf $J_{i}$ that defining the divisor corresponding to $\Sigma_{i}$. There is canonical log structure $\sN_{i}$ which is the log structure associated to the pre-log structure $\N\to\sO_{C}$ sending $1$ to $\sigma_{i}$. Now we have another canonical log structure taking into account the markings:
\begin{equation}
\sM_{C}^{\sharp} = \sM^{C/S}_{C}\oplus_{\sO_{C}}\sN_{1}\oplus_{\sO_{C}}\sN_{2}\oplus_{\sO_{C}}\cdots\oplus_{\sO_{C}}\sN_{n}.
\end{equation}   
We also have a log smooth map $\pi^{\sharp}:(C,\sM_{C}^{\sharp})\to(S,\sM_{S}^{C/S})$, induced by the canonical log structure without markings. Consider any log morphism $\pi:(C,\sM_{C})\to(S,\sM_{S})$ such that
\begin{enumerate}
 \item The underlying map $C\to S$ is family of prestable curves with $n$ markings $\{\Sigma_{i}\}$;
 \item There is another log structure $\sM_{C}'$ on $C$, such that \[\sM_{C}=\sM_{C}'\oplus\sum_{i}{\sN_{i}};\]
 \item The log morphism $\pi$ is induced by a log smooth, integral map $(C,\sM_{C}')\to(S,\sM_{S})$.
\end{enumerate}
The universal property as in (\ref{diag:curve-univ}) implies that there are a unique pair of morphisms of log structures $\sM_{C}^{\sharp}\to \sM_{C}$ and $\sM_{S}^{C/S}\to\sM_{S}$ fitting in to the following catesian diagram:
\begin{equation}\label{Diag:CurveCanLog}
\xymatrix{
(C,\sM_{C}) \ar[r] \ar[d] & (C,\sM_{C}^{\sharp}) \ar[d]\\
(S,\sM_{S}) \ar[r] & (C,\sM_{S}^{C/S}).
}
\end{equation}

\begin{defn}
A triple $(C\to S,\{\Sigma_{i}\}_{i=1}^{n},\sM^{C/S}_{S}\to\sM_{S})$ is called a {\em genus $g$, $n$-pointed log prestable curves over $S$}, if 
\begin{enumerate}
 \item $(C\to S,\{\Sigma_{i}\}_{i=1}^{n})$ is a usual genus $g$, $n$-pointed prestable curve over $S$;
 \item $\sM^{C/S}_{S}\to\sM_{S}$ is a morphism of fine log structures, where $\sM^{C/S}_{S}$ is the canonical log structure as in (\ref{Diag:CurveCanLog}).
\end{enumerate}
Denote by $\sM_{C}$ the log structure on $C$ given by (\ref{Diag:CurveCanLog}). In the rest of the paper, we will use $(C/S,\sM_{S})$ to denote the log prestable curve over $S$, when there is no confusion about the markings and the canonical log structure on the base. 

An arrow $(C_{1}\to S_{1}, \sM_{S_{1}})\to (C_{2}\to S_{2}, \sM_{S_{2}})$ between two log prestable curves is a cartesian diagram of log schemes:
\[
\xymatrix{
(C_{1}, \sM_{C_{1}}) \ar[r] \ar[d] & (C_{2}, \sM_{C_{2}}) \ar[d]\\
(S_{1}, \sM_{S_{1}}) \ar[r] & (S_{2}, \sM_{S_{2}}).
}
\]
where the bottom arrow is strict as in Section \ref{ss:DefLogStr}.
\end{defn}

\subsection{Minimality condition}
We introduce the minimality condition on log pre-stable curves as in \cite[3.7]{Kim}. 

\begin{defn}\label{def:MinCond}
A log prestable curve $(C\to S,\sM^{C/S}_{S}\to\sM_{S})$ over $S$ is called {\em minimal} if it satisfies the following two conditions:
\begin{enumerate}
 \item the log structure $\sM_{S}$ is free, and there is no proper free sub-log structure of $\sM_{S}$ containing the image of $\sM_{S}^{C/S}$;
 \item for any $s \in S$ and irreducible elements $b \in \overline{\sM}_{S,\overline{s}}$, there is an irreducible elements $a\in \overline{\sM}_{S,\overline{s}}^{C/S}$ such that $a \mapsto l\cdot b$ for some positive integer $l$.
\end{enumerate}
\end{defn}

\begin{rem}\label{rem:min-open}
It was proved in \cite[5.3.2]{Kim} that the minimality condition is an open condition.
\end{rem}

\subsection{Extended log structure on curves}
For the purpose of degeneration formula, we need to introduce the extended log structure on minimal log prestable curves.

\begin{defn}\label{def:ExtMinCurve}
A log prestable curves $(C\to S,\{\Sigma_{i}\}_{i=1}^{n},\sM^{C/S}_{S}\to\sM_{S})$ over $S$ is called {\em extended minimal log prestable}, if there is a sub-log structure $\sM_{S}'\to\sM_{S}$, such that
 \begin{enumerate}
  \item the arrow $\sM^{C/S}_{S}\to\sM_{S}$ factors through $\sM_{S}^{C/S}\to \sM_{S}'$;
  \item the log prestable curve $(C\to S,\sM_{S}')$ is minimal;
  \item locally we have a chart for $\sM^{C/S}_{S}\to\sM_{S}'\to\sM^{S}$ as follows:
\[
\xymatrix{
\sM^{C/S}_{S} \ar[rr] && \sM_{S}' \ar[rr] && \sM_{S}\\
\N^{m} \ar[rr] \ar[u] && \N^{n} \ar[u] \ar[rr]^{(id,0)} && \N^{n}\oplus\N^{n'} \ar[u].
}
\]
 \end{enumerate}
\end{defn}

\begin{prop}\label{prop:StackMinCurve}
The stack $\fM^{ext}_{g,n}$ parameterizing genus $g$, $n$-pointed extended minimal log prestable curves is an algebraic stack.
\end{prop}
\begin{proof}
Consider the log stack $\sL og_{\fM_{g,n}}$ as in Section \ref{s:LogStack}. By the discussion in Section \ref{s:LogStack}, the stack $\sL og_{\fM_{g,n}}$ parameterizing all log smooth curves. By Remark \ref{rem:min-open}, the extended minimal log prestable curves forms an open substack of $\sL og_{\fM_{g,n}}$. This proves the statement. 
\end{proof}

\subsection{Curves with disjoint components}

Let $V$ be a finite set with two weight functions $g:V\to \N$ and $n:V\to \N$. 
\begin{defn}
Consider a family of curves $C\to S$, such that $C$ is the disjoint union of $C_{v}$ for all $v\in V$, where $C_{v}\to S$ is a usual family of prestable curves of genus $g(v)$ and $n(v)$ marked points. We call such $C\to S$ the {\em prestable curves with disjoint components of data $V$}.
\end{defn}

Note that there is a canonical log structure on $S$ given by 
\[\sM_{S}^{C/S}=\sum_{v\in V}\sM_{S}^{C_{v}/S}.\]
This induces a log structure on $C_{v}$ by (\ref{Diag:CurveCanLog}), hence a canonical log structure $\sM_{C}^{\sharp}$ on $C$. Similarly, we have the following definition:

\begin{defn}
Let $C\to S$ be a family of prestable curves with disjoint components of data $V$. A tuple $(C\to S,\sM^{C/S}_{S}\to\sM^{S})$ over $S$ is called a { \em minimal log prestable curve with disjoint components over $S$}, if the morphism of log structures $\sM^{C/S}_{S}\to\sM_{S}$ satisfies the two conditions in Definition \ref{def:MinCond}. It is called {\em extended minimal log prestable curves with disjoint components}, if $\sM^{C/S}_{S}\to\sM_{S}$ satisfies the three conditions in Definition \ref{def:ExtMinCurve}.
\end{defn}

Note that the stack parameterizing prestable curve of disjoint components of data $V$ is given by 
\[\fM_{V}=\prod_{v\in V}{\fM_{g(v),n(v)}}\]
where $\fM_{g(v),n(v)}$ is the algebraic stack of genus $g(v)$ prestable curves with $n(v)$-marked points. Note that $\fM_{V}$ has a canonical log structure, which comes from the canonical log structure of each $\fM_{g(v),n(v)}$. Thus, we can view $\fM_{V}$ as a log stack. 

\begin{prop}\label{prop:StackDisjointCurve}
The stack $\fM_{V}^{ext}$ parameterizing extended minimal log prestable curves with disjoint components of data $V$ is an open substack of $\sL og_{\fM_{V}}$, hence is an algebraic stack.
\end{prop}
\begin{proof}
The proof is identical to the one for Proposition \ref{prop:StackMinCurve}. 
\end{proof}

\section{Log stable maps}\label{sec:general-target}
In this section we introduce our log structures on Jun Li's predeformable maps \cite{Jun1}. This is mainly Kim's log structure as in \cite{Kim}, but we add the standard log structures on markings and the fixed divisor on the target as in Section \ref{sss:DivLog}. Such difference is mainly for the purpose of the degeneration formula as in Theorem \ref{thm:main}.

\subsection{Kim's log structure on the target}\label{ss:LogStrTarget}
Here we gathering some of the definitions and results in \cite[section 4]{Kim}. 

\begin{defn}\label{def:FM-space}
An algebraic space $W$ over $S$ is called a {\em Fulton-Macpherson (FM) type space} if 
\begin{enumerate}
 \item $W\to S$ is a proper, flat map;
 \item for every closed point $s\in S$, \'etale locally there is an \'etale map 
 \[W_{\bar{s}}\to \mbox{Spec}k(\bar{s})[x,y,z_{1},\cdots,z_{r-1}]/(xy)\] where $x,y$ and $z_{i}$ are independent variables with the only relation $xy=0$.
\end{enumerate}

Consider a FM type space $W\to S$, if it admits a log smooth morphism 
\begin{equation}\label{equ:log-FM}
\pi:(W,\sM^{W/S}_{W})\to (S,\sM^{W/S}_{S}),
\end{equation}
then we call $W\to S$ the {\em log FM type space}. 
\end{defn}

\begin{rem}\label{rem:log-special}
Recall that in \cite{LogSS}, the log smooth morphism $\pi$ in (\ref{equ:log-FM}) is called {\em special} if:
\begin{enumerate}
 \item $\sM^{W/S}_{W}$ and $\sM^{W/S}_{S}$ are free;
 \item For any $w \in W$, the induced map $\pi^{\flat}:\overline{\pi^{*}\sM}^{W/S}_{\bar{w},S}\to \overline{\sM}^{W/S}_{\bar{w},W}$ is an isomorphism if $w$ is a smooth point of $W\to S$, or part of the cocatesian diagram 
 \[
 \xymatrix{
 \N \ar[rr]^{\Delta: e_{}\mapsto e_{1}+e_{2}} \ar[d]_{h} && \N^{2} \ar[d] \\
 \CharM_{\pi(\bar{w}),S}^{W/S} \ar[rr] && \CharM^{W/S}_{\bar{w},W}.
 }
 \] 
 when $w$ is in the singular locus, where $h(e_{})$ is an irreducible element.
 \item There is a bijection induced by the above diagram
 \[\mbox{Irr}W_{\bar{s}}^{sing}\to \mbox{Irr}\CharM_{\bar{s},S}^{W/S},\]
 where $\mbox{Irr}W_{\bar{s}}^{sing}$ is the set of irreducible components of the singular locus of $W_{\bar{s}}$ and $\mbox{Irr}\CharM_{\bar{s},S}^{W/S}$ is the set of irreducible elements of $\CharM_{\bar{s},S}^{W/S}$.
\end{enumerate}
Such log structure on $W$ and $S$ is called the {\em canonical log structure} associated to the log FM type space $W\to S$.
\end{rem}

\begin{defn}\label{def:log-twist-FM}
A pair $(W\to S,\sM_{S}^{W/S}\to\sM_{S})$ is called an {\em extended log twisted FM type space over $S$}, if $W\to S$ is a log FM type space, and the morphism of log structures $\sM_{S}^{W/S}\to\sM_{S}$ on $S$ is simple. Namely, \'etale locally at any point $s\in S$, there is a commutative diagram of charts:
\[
\xymatrix{
\sM_{S}^{W/S} \ar[rrrr] &&&& \sM_{S} \\
\N^{m}\ar[rr]^{(r_{1},\cdots,r_{m})} \ar[u]^{\theta^{W/S}} && \N^{m} \ar[rr]^{(id,0)\ \ } && \N^{m}\oplus\N^{n} \ar[u]_{\theta} 
}
\]
where the first bottom map is the diagonal map, and the maps $\theta^{W/S}$ and $\theta$ induce an isomorphism from $\N^{m}$ to $\overline{\sM}_{S,\bar{s}}^{W/S}$, and $\N^{m}\oplus\N^{n}$ to $\overline{\sM}_{S,\bar{s}}$ respectively.

The integer $r_{i}$ is called the {\em log twisting index}. If $n=0$, then the pair $(W\to S,\sM_{S}^{W/S}\to\sM_{S})$ is called a {\em unextended log twisted FM type space} or just {\em log twisted FM type space}. As in the case of curves, the log structure on $W$ is given by 
\[\sM_{W}=\pi^{*}(\sM_{S})\oplus_{\pi^{*}(\sM_{S}^{W/S})}\sM_{W}^{W/S}.\]
\end{defn}

Note that we have a log smooth, integral map $(W,\sM_{W})\to (S,\sM_{S})$.

\begin{defn}
We call $(W\to S, D)$ a smooth pair over $S$, if $D \hookrightarrow W\to S$ is a smooth divisor of $W$ over $S$. 
\end{defn}

Similarly we have the following:
\begin{defn}\label{def:LogTwistedFMPair}
A triple $(W\to S, D, \sM_{S}^{W/S}\to\sM_{S})$ is called a {\em (un)extended log twisted smooth pair over $S$}, if $(W\to S, \sM_{S}^{W/S}\to\sM_{S})$ is a (un)extended log twisted FM type space, and $(W\to S, D)$ is a smooth pair. We denote $\sM^{D}$ to be the log structure associated to the divisor $D$, thus the log structure on $W$ is defined to be 
\[\sM_{W}=\pi^{*}(\sM_{S})\oplus_{\pi^{*}(\sM_{S}^{W/S})}\sM_{W}^{W/S}\oplus_{\sO_{W}}\sM^{D}.\]
Note that we have a log smooth morphism $(W,\sM_{W})\to (S,\sM_{S})$.
\end{defn}

\begin{notation}
For simplicity, we will use $(W\to S,\sM_{S})$ to denote the extended log twisted FM type space or smooth pair when there is no confusion about the map $\sM_{S}^{W/S}\to\sM_{S}$ and the divisor $D$.
\end{notation}

\begin{defn}\label{def:target-arrow}
Consider two extended log twisted FM type spaces or smooth pairs $(W_{i}\to S_{i},\sM_{S_{i}})$ for $i=1,2$. An arrow $(W_{1}\to S_{1},\sM_{S_{1}})\to (W_{2}\to S_{2},\sM_{S_{2}})$ is a cartesian diagram of log schemes as follows:
\[
\xymatrix{
W_{1} \ar[r] \ar[d] & W_{2} \ar[d] \\
S_{1} \ar[r] & S_{2} 
}
\]
where the bottom arrow is strict.
\end{defn}

Consider a stack $\sB$, which parametrizes one of the following objects:
\begin{enumerate}
 \item a pair $(W\to S, W\to X\times S)$ such that $W\to S$ is log FM type space, and $X$ is a fixed scheme;
 \item a tuple $(W\to S,D,W\to X\times S)$ same as above but with a smooth divisor $D$ in $W$.
\end{enumerate}
We assume that $\sB$ is an algebraic stack, and $\sU\to \sB$ is its universal family. Thus $\sB$ has a log structure given by $\sM_{\sB}^{\sU/\sB}$ the canonical log structure given by this family. We view $\sB$ as a log stack with this canonical log structure. 

\begin{lem}\label{lem:target-stack-alg}
With the above assumptions, the stack $\sB^{tw}$ (respectively $\sB^{etw}$), which parametrizes (respectively extended) log twisted FM type space (or smooth pair) is an algebraic stack.
\end{lem}
\begin{proof}
By Definition \ref{def:log-twist-FM} and \ref{def:LogTwistedFMPair} and \cite{LogStack}, the stack $\sB^{tw}$ (respectively $\sB^{etw}$) is an open substack of $\sL og_{\sB}$, hence an algebraic stack. 
\end{proof}

\begin{rem}\label{rem:remove-free}
We have a natural map $\tau: \sB^{etw}\to \sB^{tw}$ by removing the extended log structures. In this way, we view $\sB^{etw}$ as a log stack over $\sB^{tw}$. By \cite[3.19]{LogStack}, there is an open immersion $\sB^{tw}\to \sB^{etw}$.
\end{rem}

\begin{rem}
In Section \ref{sec:target}, we will introduce the stacks of expanded pairs and expanded degeneration. Those stacks will satisfy the conditions for $\sB$ as above.
\end{rem}
\subsection{Log Stable Maps}
Next we introduce Kim's definition for log stable maps \cite[section 5]{Kim}, with the only difference that we have the standard log structure given by markings on the source curve and the smooth divisors on the target.

\begin{defn}\label{def:LogPreMap}
A log morphism 
\[\xi: \ (f:(C,\sM_{C})\to(W,\sM_{W}))/(S,\sM_{S})\]
is called a {\em $(g,n)$ log prestable map over $S$} if:
\begin{enumerate}
\item $(C\to S, \sM_{S})$ is a genus $g$, $n$-pointed, extended minimal log prestable curve.
\item $(W\to S, \sM_{S})$ is an extended log twisted FM type space.
\item (\textbf{Corank condition}) For every point $s\in S$, the rank of $\Coker(\sM_{S,s}^{W/S}\to \sM_{S,s})$ coincides with the number of non-distinguished nodes on $C_{\bar{s}}$.
\item $f:(C,\sM_{C})\to(W,\sM_{W})$ is a log morphism.
\item (\textbf{Log admissibility}) the morphism of log structures $f^{\flat}:f^{*}\sM_{W}\to \sM_{C}$ is simple at every distinguished node.
\end{enumerate}
In case the source has disjoint components, the definitions are the same, except we put the conditions for genus and marked points on each connected components separately. Here a node is called {\em distinguished} if it maps to the singular locus of $W$, otherwise is called {\em non-distinguished}. 
\end{defn}

\begin{defn}\label{def:log-map-arrow}
An arrow $\xi\to \xi'$ between two log prestable maps is a log commutative diagram:
\[
\xymatrix{
(C,\sM_{C}) \ar[dd] \ar[rr]^{} \ar[rd] && (W,\sM_{W}) \ar[dd] \ar[ld] \\
  &(S,\sM_{S}) \ar[dd]& \\
(C',\sM_{C'}) \ar@{->}'[r][rr]^{} \ar[rd] && (W',\sM_{W'}) \ar[ld] \\
  &(S',\sM_{S'})&
  }
\]
where the top and bottom triangles correspond to $\xi$ and $\xi'$ respectively, the three side squares are all cartesian of log schemes, and the vertical arrow $(S,\sM_{S})\to (S',\sM_{S'})$ is strict.
\end{defn}

\begin{rem}
\begin{enumerate}
\item The corank condition implies that the log structure on the base is coming from the singular loci of $W$ and the non-distinguished nodes. It was proved in \cite[5.3.2]{Kim} that corank condition is an open condition.
\item It was shown in \cite[5.1.2]{Kim} that with conditions (1) - (4), the log admissibility forces the underlying maps to be predeformable in the sense of \cite{Jun1}. 
\end{enumerate}
\end{rem}

\begin{rem}\label{rem:LogAtSing}
For later use, we would like to give a short description of log prestable maps at the distinguished nodes as in \cite[5.2.3]{Kim}. Assume that $S=\Spec k$ is a geometric point. Let $D'\in W$ be a connected singular locus, and $p_{1},\cdots,p_{k}$ be the distinguished nodes that map to $D'$ via $f$, and $c_{1},\cdots,c_{k}$ is the corresponding contact order at each $p_{i}$. Denote by
\[r=l.c.m(c_{1},\cdots,c_{k}) \mbox{\ \ \ and\ \ \ }l_{i}=r/c_{i}.\] 
Let $m$ and $m'$ be the number of disjoint singular loci and non-distinguished nodes. Then locally at $p_{i}$ and $f(p_{i})$ we have charts given by 
\[(\N^{2}\oplus_{\N}\N)\oplus\N^{m+m'-1} \to \sM_{C}\]
and
\[(\N^{2}\oplus_{\N}\N)\oplus\N^{m+m'-1} \to \sM_{W},\]
where the first amalgamated sum are given by the push-out of:
\[
\xymatrix{
\N \ar[rr]^{e\mapsto l\cdot e'} \ar[d]_{e\mapsto a+b} && \N \\
\N^{2}
}
\]
and the second amalgamated sum are given by the push-out of:
\[
\xymatrix{
\N \ar[rr]^{e\mapsto r\cdot e'} \ar[d]_{e\mapsto x+y} && \N \\
\N^{2}
}
\]

On the level of charts, the morphism of log structures $f^{*}\sM_{W}\to \sM_{C}$ is given by identity on $\N^{m+m'-1}$ and the following diagram:
\[
\xymatrix{
\N^{2}\oplus_{\N}\N \ar[rrrr]^{(x,y)\mapsto (c_{i}\cdot a,c_{i}\cdot b)} &&&& \N^{2}\oplus_{\N}\N\\
&&\N \ar[ull]^{e'\mapsto (0,e')} \ar[rru]_{e'\mapsto (0,e')}&&
}
\]
We remark that such description also holds when the source curve has disjoint components.
\end{rem}

\begin{rem}\label{rem:MapToSmoothTarget}
As pointed out in \cite[5.3.1]{Kim}, if the target $W\to S$ is a family of smooth projective variety over $S$, then any stable map to such target is an underlying usual stable map equipped with the canonical log structure from its underlying curve. In this case, all nodes are non-distinguished.
\end{rem}

Let $(W\to S,D)$ be a smooth pair. Denote by $N$ the index set of all marked points with $|N|=n$, and $M\subset N$ a subset. Let $\Bc=(c_{j})_{j\in M}$ be the tuple of positive integers, which will be the assigned tangency multiplicity at the marked points indexed by $M$. 

\begin{defn}\label{def:relative-log-map}
A log map
\[(f:(C,\sM_{C})\to(W,D,\sM_{W}))/(S,\sM_{S})\]
is called a {\em $(g,n)$ relative log prestable map over $S$ with tangency condition $\Bc$} if:
\begin{enumerate}
\item By removing the log structure associated to $D$, the map $f$ is a $(g,n)$ log prestable map over $S$.
\item The underlying map of $f$ is non-degenerate, namely no component of $C$ maps into $D$ via the map $f$.
\item The marked points indexed by $M$ has assigned tangency multiplicity $\Bc$ such that $f_{*}[C].D=\sum_{j\in M}c_{j}$.
\end{enumerate}
Similarly when the source curves has disjoint components, we put the conditions for genus, marked points and the intersection multiplicity on each connected components separately.
\end{defn}

\begin{rem}\label{rem:LogOnMarking}
For any $j\in M$, with log structure put on marked points $\Sigma_{j}$ and the smooth divisor $D$, the underlying map induces the map $\sM^{\Sigma}\to f^{*}(\sM^{D})$, which \'etale locally has the following chart:
\[
\xymatrix{
f^{*}(\sM^{D}) \ar[rr] &&  \sM^{\Sigma}\\
\N \ar[u] \ar[rr]^{\log x\mapsto c_{j}\cdot\log \delta} && \N \ar[u],
}
\]
where $x$ and $\delta$ is the local coordinates whose vanishing correspond to $\Sigma$ and $D$ respectively. Since the underlying map is non-degenerate by Definition \ref{def:relative-log-map}(2), the morphism of log structures $f^{\flat}$ along the marked points with tangency condition is uniquely determined by the underlying map. Note that the tangency condition with log structures is an open condition, for example see \cite[Precise location]{Chen}.
\end{rem}

\begin{defn}\label{def:StabilityCond}\cite[5.2.4]{Kim}
Let $W\to S$ be equipped with a map $\pi_{X}:W\to X$ from $W$ to a scheme. A (relative) log prestable map defined as above is called {\em log stable} if for any point $s\in S$, the group of automorphisms $(\sigma,\tau)$ is finite where
\begin{enumerate}
 \item $\sigma$ is an automorphism of $(C\to S,\sM_{S})_{\bar{s}}$ preserving the marked points.
 \item $\tau$ is an automorphism of $(W\to S,\sM_{S})_{\bar{s}}$ preserving $W_{\bar{s}}\to X$.
 \item $\tau\circ f_{\bar{s}}=f_{\bar{s}}\circ \sigma$.
\end{enumerate}
The automorphism of the underlying maps can be defined similarly without considering the log structures.
\end{defn}

\begin{rem}
It was shown in \cite[6.2.4]{Kim} that $f$ is log stable if and only if the underlying map $\underline{f}$ is stable.
\end{rem}

\begin{defn}
A (relative) log stable map $f$ is said to have curve class $\beta\in H_{2}(X,\Z)$, if $(\pi_{X}\circ f)_{*}[C_{\bar{s}}]=\beta$ for every point $s \in S$.
\end{defn}

Next we consider the case when the source curves has disjoint components. We fix a smooth pair $(W\to S, D)$ and a map $\pi_{X}: W\to X\times S$. Follow \cite{Jun1} and \cite{AF}, we introduce the following refined data for disconnected relative log stable maps.
\begin{defn}\label{def:WeightGraph}
An {\em admissible weighted graph} $\Xi$ is a collection of vertices $V(\Xi)$, legs $L(\Xi)$ and roots $R(\Xi)$ but with no edges, coupled with the following data:
\begin{enumerate}
\item each vertex $v \in V(\Xi)$ is assigned a non-negative integers $g(v)$, and and a curve class $\beta(v)\in H_{2}(X;\Z)$;
\item each root $j\in R(\Xi)$ is assigned a positive integer $c_{j}$;
\end{enumerate}
Furthermore, we require the graph $\Xi$ to be relatively connected, i.e. either $V(\Xi)$ contains a single element or each vertex in $V(\Xi)$ has at least one root attached to it.
\end{defn}

For each vertex $v\in V(\Xi)$, we denote $R_{v}$ and $L_{v}$ to be the sets of roots and legs that attached to $v$.
\begin{defn}
A log morphism 
\[(f:(C,\sM_{C})\to(W,D,\sM_{W}))/(S,\sM_{S})\]
is called a {\em $\Xi$-relative log stable maps over $S$} if
\begin{enumerate}
 \item $C$ is a disjoint union of $\{C_{v}\}_{v\in V(\Xi)}$ such that each $(C\to S, \sM_{S})$ is an extended minimal log prestable curves with disjoint components over $S$ with genus and marked points on each connected component given by $g(v)$ and $R_{v}\cup L_{v}$ respectively.
 \item The curve class for each $f|_{C_{v}}$ is $\beta(v)$.
 \item The map $f$ is a relative log prestable map with assigned tangency multiplicity $c_{j}$ along the marked points indexed by $j\in R_{v}$ such that $\beta(v).D=\sum_{j\in R_{v}}c_{j}$.
 \item The map $f$ satisfies the stability condition in Definition \ref{def:StabilityCond}.
\end{enumerate} 
\end{defn}

\subsection{Stack of log stable maps}\label{ss:map-stack}
We fix the stack $\sB$ as in Lemma \ref{lem:target-stack-alg}. Denote by $\sU\to \sB$ the universal family over $\sB$. Define $\sK^{log}_{g,n}(\sU/\sB,\beta)$ (resp. $\sK^{log}_{\Xi}(\sU/\sB)$) to be the stack parameterizing (resp. relative) log stable maps of genus $g$, $n$ marked points with curve class $\beta$ (resp. with data $\Xi$). Similarly, we define $\sK_{g,n}(\sU/\sB,\beta)$ and $\sK_{\Xi}(\sU/\sB)$ to be the stack parameterizing the underlying structure of the log stable maps. For simplicity, we use the short hand notation $\sK_{\sB}^{log}$ for $\sK^{log}_{g,n}(\sU/\sB,\beta)$ or $\sK^{log}_{\Xi}(\sU/\sB)$, and $\sK_{\sB}$ for $\sK_{g,n}(\sU/\sB,\beta)$ or $\sK_{\Xi}(\sU/\sB)$. 

\begin{rem}\label{rem:predeformable}
The maps parametrized by $\sK_{\sB}$ is called stable predeformable maps as in \cite{Jun1}. Consider a fixed stable predeformable map. It was shown in \cite[6.3.1(2)]{Kim} that there exists (not necessarily unique) log stable maps whose underlying structures is identical to the given predeformable map. 
\end{rem}

Note that we have a map 
\[\sK^{log}_{\sB}\to \fM^{log}\times_{LOG}\sB^{etw}\]
by sending a map to its source log curve and target. Here $\fM^{log}$ is the stack parametrizing the source log curves, which can be either $\fM^{ext}_{g,n}$ as in Proposition \ref{prop:StackMinCurve} or $\sM_{V}$ as in Proposition \ref{prop:StackDisjointCurve}, and the product is in the category of log stacks. Similarly, we have the map 
\[ \sK_{\sB}\to \fM\times \sB\]
where $\fM$ is the stack parameterizing corresponding curves without log structures. The following result is proved in \cite{Kim}.

\begin{thm}\label{thm:LogStackForMap}
If the stack $\sK_{\sB}$ is a proper DM-stack, then so is $\sK^{log}_{\sB}$.
\end{thm}
\begin{proof}
When the target is log FM type spaces, it was shown in \cite[5.3.2]{Kim} that $\sK^{log}_{\sB}$ is a locally closed substack of $\sK_{\sB}\times_{\fM\times \sB}\fM^{log}\times_{LOG}\sB^{etw}$, therefore an algebraic stack. When the target is smooth pairs, the map of log structures on the markings are uniquely determined by the underlying structure as in Remark \ref{rem:LogOnMarking}. Thus the same proof as in \cite{Kim} still holds. The properness is proved in \cite[6.3.2]{Kim}. Finally, since we are working over an algebraic field of characteristic zero, the finiteness of automorphisms implies the diagonal of $\sK^{log}_{\sB}$ is unramified.
\end{proof}

\begin{rem}\label{rem:StackSmTarget}
Let us consider the case where the target is a smooth projective variety $X$. By Remark \ref{rem:MapToSmoothTarget}, the stack $\sK^{log}_{g,n}(X,\beta)$ is the usual stack parameterizing stable maps to $X$ from genus $g$, $n$-marked curves with curve class $\beta$. The log structure on $\sK^{log}_{g,n}(X,\beta)$ is the canonical one corresponding to the smoothing of its universal curves.
\end{rem}

\begin{rem}\label{rem:ForgetNondistNode}
Note that the structure arrow $\sK_{B}^{log}\to \sB$ factor through $\sK^{log}_{\sB}\to\sB^{etw}\to\sB^{tw}$. In fact, it is convenient to view $\sK^{log}_{\sB}$ as a log stack over $\sB^{tw}$. On the one hand, relative to $\sB^{tw}$ still gives the information of the log structure along each singular locus of the target. On the other hand, we get rid of the information of log structures from the non-distinguished nodes. Later we will construct a perfect obstruction theory relative to $\sB^{tw}$
\end{rem}

\begin{rem}\label{rem:BaseChange}
Consider any $T\to \sB$. Denote by $\sU_{T}=\sU\times_{\sB}T$ the universal family over $T$, and $p:\sU_{T}\to \sU$ to be the projection. Let $\beta'$ be the curve class in the fiber of $\sU_{T}$ induced by $\beta$. Denote by $\Xi'$ the data obtained by replacing $\beta$ with $\beta'$ in $\Xi$. Let $\sK_{T}$ to be either $\sK_{g,n}(\sU_{T}/T,\beta)$ or $\sK_{\Xi}(\sU_{T}/T)$. Then we say that $\sK_{\sB}$ is {\em well-behaved under base change}, if we have the following catesian diagram:
\[
\xymatrix{
\sK_{T} \ar[r] \ar[d] & \sK_{\sB} \ar[d]\\
T \ar[r] & \sB.
}
\]
Here the top arrow is given by sending a log map $f$ to $p\circ f$. With this assumption, it is not hard to see that the log stack $\sK_{\sB}^{log}$ is also well-behaved under base change, namely we have the following catesian diagram:
\[
\xymatrix{
\sK^{log}_{T} \ar[r] \ar[d] & \sK_{\sB}^{log} \ar[d] \\
T^{tw} \ar[r] \ar[d] & \sB^{tw} \ar[d] \\
T \ar[r] & \sB.
}
\]
\end{rem}

\section{The Obstruction theory and virtual fundamental classes}\label{sec:fundamental-class}
In this section, we use the method developed in \cite[Appendix C]{AF} to define the perfect obstruction theory for the log stable maps. Denote by $\BL^{log}$ and $\BL^{G}$ the Olsson's and Gabber's log cotangent complex respectively in the sense of \cite{LogCot}. We use $\BL$ for the usual cotangent complex. We refer to \cite{LogCot} for properties of these log cotangent complexes.

\subsection{Perfect obstruction theory for (relative) log stable maps}\label{ss:ConstructVir}
We adopt the notations in Section \ref{ss:map-stack}. Denote by $\sC$ the universal curve over $\sK^{log}_{\sB}$. Let $\sU^{tw}$ and $\sU^{etw}$ be the universal families of targets over $\sB^{tw}$ and $\sB^{etw}$ respectively. By Remark \ref{rem:ForgetNondistNode}, we have the following commutative diagram of log stacks:
\begin{equation}\label{diag:total-diag}
\xymatrix{
\sC \ar@/^/[rrrd]^{f'} \ar[rd]^{f} \ar@/_/[rdd]_{q} &&&\\
&\sU^{etw}_{\sK} \ar[r] \ar[d]^{\psi} & \sU^{etw} \ar[d] \ar[r] & \sU^{tw} \ar[d]\\
&\sK^{log}_{\sB} \ar[r] &\sB^{etw} \ar[r] & \sB^{tw},
}
\end{equation}
where the squares are all cartesian squares of log stacks, and $f$ is the universal (relative) log stable map. By the functoriality of Gabber's log cotangent complex \cite[8.24]{LogCot}, we have triangles
\begin{equation}\label{eq:UpCot}
f^{*}\BL^{G}_{\sU^{etw}_{\sK}/\sU^{tw}}\to\BL^{G}_{\sC/\sU^{tw}}\to\BL^{G}_{f}
\end{equation}
and
\begin{equation}\label{eq:LowCot}
f^{*}\BL^{G}_{\sU^{etw}_{\sK}/\sK_{\sB}^{log}}\to\BL^{G}_{\sC/\sK_{\sB}^{log}}\to\BL^{G}_{f}.
\end{equation}
By \cite[8.29]{LogCot} and the definition of the log stable maps, we have 
\[\BL^{G}_{\sU^{etw}_{\sK}/\sU^{etw}}\cong \BL^{log}_{\sU^{etw}_{\sK}/\sU^{etw}}, \ \ \ \BL^{G}_{\sU^{etw}_{\sK}/\sK_{\sB}^{log}}=\BL^{log}_{\sU^{tw}_{\sK}/\sK_{\sB}^{log}},\] 
and 
\[\BL^{G}_{\sC/\sK_{\sB}^{log}}=\BL^{log}_{\sC/\sK_{\sB}^{log}}.\] 
Thus, by (\ref{eq:LowCot}) we have the following triangle
\begin{equation}\label{equ:LowLogCot}
f^{*}\BL^{log}_{\sU^{etw}_{\sK}/\sK_{\sB}^{log}}\to\BL^{log}_{\sC/\sK_{\sB}^{log}}\to\BL^{log}_{f}.
\end{equation}
This gives a canonical isomorphism $\BL^{log}_{f}\cong \BL^{G}_{f}$. Denote by $\BL_{\square}$ the log cotangent complex $\BL^{log}_{f}$.

\begin{lem}\label{lem:PullBackCotCx}
$\BL^{log}_{\sU^{etw}_{\sK}/\sU^{tw}} \cong  \psi^{*}\BL^{log}_{\sK_{\sB}^{log}/\sB^{tw}}$
\end{lem}
\begin{proof}
By \cite[3.2]{LogCot}, the log cotangent complex $\BL_{\sK^{log}_{\sB}/\sB^{tw}}^{log}$ is the usual cotangent complex of the map of algebraic stacks:
\[
h : \sK^{log}_{\sB}\to \sL og_{\sB^{tw}}
\] 
where the above map is induced by the map of log stacks $\sK^{log}_{\sB}\to \sB^{tw}$. By our construction of log stable maps, the image of $h$ lies in the open sub-stack $\sB^{etw}\subset \sL og_{\sB^{tw}}$. Thus, we have
\[
\BL_{\sK^{log}_{\sB}/\sB^{tw}}^{log}\cong \BL_{\sK^{log}_{\sB}/\sB^{etw}}.
\]
Similarly, we have 
\[
\BL^{log}_{\sU^{etw}_{\sK}/\sU^{tw}}\cong \BL_{\sU^{etw}_{\sK}/\sU^{etw}}.
\]
Note that the squares in (\ref{diag:total-diag}) are cartesian squares of log stacks. Now the statement follows from the base change of usual cotangent complex. 
\end{proof}

By Lemma \ref{lem:PullBackCotCx} and (\ref{eq:UpCot}), we have an arrow $\BL_{\square}[-1]\to q^{*}\BL_{\sK^{log}_{\sB}/\sB^{tw}}^{log}$. Let $\omega_{q}$ be the dualizing complex of $q$, then it is a line bundle sitting in degree $-1$. By tensoring with $\omega_{q}$, we have 
\[\BL_{\square}\otimes \omega_{q}[-1]\to q^{*}\BL_{\sK^{log}_{\sB}/\sB^{tw}}^{log}\otimes \omega_{q}.\]
Since $q^{!}(\BL_{\sK^{log}_{\sB}/\sB^{tw}}^{log})=q^{*}\BL_{\sK^{log}_{\sB}/\sB^{tw}}^{log}\otimes \omega_{q}$, and $q^{!}$ is left adjoint to $q_{*}$, we have an arrow
\begin{equation}\label{equ:obs1}
q_{*}(\BL_{\square}\otimes \omega_{q}[-1])\to \BL_{\sK^{log}_{\sB}/\sB^{tw}}^{log}.
\end{equation}
Denote by $\E_{\sK_{\sB}^{log}/\sB^{tw}}=q_{*}(\BL_{\square}\otimes \omega_{q}[-1])$. A standard argument yields: 
\begin{equation}\label{equ:obs2}
\E_{\sK_{\sB}^{log}/\sB^{tw}}\cong(q_{*}(\BL_{\square})^{\vee})^{\vee}[-1].
\end{equation} 

\begin{prop}
The arrow $\E_{\sK_{\sB}^{log}/\sB^{tw}} \to \BL_{\sK^{log}_{\sB}/\sB^{tw}}^{log}$ in derived category gives a perfect obstruction theory relative to $\sB^{etw}$.
\end{prop}
\begin{proof}
Note that $\BL_{\square}=\BL^{G}_{f}$. The obstruction theory argument follows from the deformation theory of Gabber's log cotangent complex \cite[8.31]{LogCot}. Since the families are log smooth, by \cite[1.1(iii)]{LogCot} we have
\[f^{*}\BL^{G}_{\sU^{etw}_{\sK}/\sK_{\sB}^{log}}\cong f^{*}\BL^{log}_{\sU^{etw}_{\sK}/\sK_{\sB}^{log}} \cong (f')^{*}\Omega_{\sU^{tw}/\sB^{tw}}^{log}\]
and 
\[\BL^{G}_{\sC/\sK_{\sB}^{log}} \cong \BL^{log}_{\sC/\sK_{\sB}^{log}} \cong \Omega_{\sC/\sK_{\sB}^{log}}^{log}.\] 
By (\ref{eq:LowCot}), $\BL_{\square}$ is the cone of the following locally free sheaves:
\begin{equation}\label{eq:ConeOfCot}
(f')^{*}\Omega_{\sU^{tw}/\sB^{tw}}^{log} \to \Omega_{\sC/\sK_{\sB}^{log}}^{log}.
\end{equation}
Since $q$ is proper and flat of relative dimension $1$, we have $\E_{\sK_{\sB}^{log}/\sB^{tw}}=\mbox{R}q_{*}(\BL_{\square}\otimes \omega_{q}[-1])$ is perfect in $[-1,0]$. 
\end{proof}

By \cite{Behrend-Fantechi} and \cite{Kresch}, the perfect obstruction theory $\E_{\sK_{\sB}^{log}/\sB^{tw}}$ gives a virtual fundamental class of $\sK^{log}_{\sB}$ relative to $\sB^{etw}$. Denote by $[\sK_{\sB}^{log}]$ the resulting virtual cycle.

\begin{rem}
Consider the morphism $\tau: \sB^{etw}\to \sB^{tw}$ given in Remark \ref{rem:remove-free}. The map $\tau$ is smooth in the usual sense. Consider the following map of complexes in derived category:
\[\E_{\sK^{log}/\sB^{tw}}[-1]\to \BL_{\tau}.\]
Denote by $\F_{\sK^{log}/\sB^{tw}}$ the cone of the above morphism. By \cite[A1]{Bryan-Leung}, we have a map 
\[\F_{\sK^{log}/\sB^{tw}}\to \BL_{\sK^{log}_{\sB}/\sB^{tw}},\]
which gives a perfect obstruction theory for $\sK^{log}_{\sB}$ relative to $\sB^{tw}$. Furthermore, the virtual fundamental class induced by $\F_{\sK^{log}/\sB^{tw}}$ coincides with the one given by $\E_{\sK^{log}/\sB^{tw}}$.
\end{rem}

\subsection{Comparison of the virtual fundamental class when target is smooth}
When the target is a smooth projective variety $X$ over $S=\Spec\C$, we assume that $X$ and $S$ are log schemes with trivial log structure. We have already seen in Remark \ref{rem:StackSmTarget} that the stack $\sK=\sK^{log}_{g,n}(X,\beta)$ is the stack $\fM_{g,n}(X,\beta)$ of usual stable maps with the canonical log structure given by its universal curves. We have the following universal diagram:
\[
\xymatrix{
\sC_{\sK} \ar[r]^{f} \ar[d]^{q}& X \\
\sK,
}
\]
where $\sC_{\sK}$ is the universal family of curves, and $f:\sC_{\sK}\to X$ is the universal map over $\sK$. In this case, the log cotangent complex $\BL_{\square}$ is giving by the cone of the following:
\[f^{*}\Omega_{X}\to \Omega_{\sC_{\sK}/\sK}^{log}.\]
Then $\E_{\sK/S}=q_{*}(\BL_{\square}\otimes\omega_{q}[-1])$ gives a perfect obstruction theory for $\sK$ relatively over $S^{ext}$, where $S^{ext}$ parametrizes log schemes with free log structures over $S$. 

On the one hand, let $\fM_{g,n}$ be the stack of usual genus $g$, $n$-pointed prestable curves. Denote by $\sC_{g,n}$ the universal curves of $\fM_{g,n}$, and $\sM_{\fM_{g,n}}$ the canonical log structure of $\fM_{g,n}$ as in Section \ref{ss:LogPresCurve}, which is locally free. Since $\fM_{g,n}$ can be viewed as a smooth (hence log smooth) stack with free log structures over $S$, we have a natural smooth map $\fM_{g,n}\to S^{ext}$. 

On the other hand, we have a natural map $\sK\to \fM_{g,n}$. Note that the log structure on $\sK$ is given by the pull-back of $\sM_{\fM_{g,n}}$ via this map. Therefore, we have a commutative diagram of log stacks:
\begin{equation}\label{diag:CompVirFund}
\xymatrix{
\sK \ar[d]_{g} \ar[dr]^{h_{2}} \\
\fM_{g,n} \ar[r]^{h_{1}} & S^{ext}
}
\end{equation}
Since $h_{1}$ is strict, by \cite[3.2]{LogCot} we have $\BL_{h_{1}}\cong \BL_{\fM_{g,n}}^{log}$. 

Consider the following
\[C_{g,n} \to \fM_{g,n} \to S.\]
This induces a triangle
\[q^{*}\BL_{\fM_{g,n}}^{log}\to \BL_{C_{g,n}}^{log} \to \Omega_{C_{g,n}/\fM_{g,n}}^{log},\]
hence a map $\Omega_{C_{g,n}/\fM_{g,n}}^{log}[-1]\to q^{*}\BL_{\fM_{g,n}}^{log}$. Repeat the argument in Section \ref{ss:ConstructVir}, we obtain an obstruction theory:
\[
q_{*}(\Omega_{C_{g,n}/\fM_{g,n}}^{log}\otimes\omega_{q})[-1]\to \BL_{\fM_{g,n}}^{log}.
\]
Since $h_{1}$ is smooth in the usual sense, by \cite[5.6]{LogCot} and \cite[17.9.1]{LMB} the above map is an isomorphism, namely
\begin{equation}\label{eq:CotLogCurve}
\BL_{h_{1}}\cong q_{*}(\Omega_{C_{g,n}/\fM_{g,n}}^{log}\otimes\omega_{q})[-1].
\end{equation}

By \cite{Behrend-Fantechi}, we have a perfect obstruction theory 
\[q_{*}(f^{*}\Omega_{X}\otimes\omega_{q})\to \BL_{\sK/\fM_{g,n}}\]
for $\sK$ relative to $\fM_{g,n}$. This induces the Behrend-Fantechi virtual fundamental class for $\sK$. By (\ref{diag:CompVirFund}), we have the following composition of maps
\begin{equation}\label{eq:RelativeToLog}
q_{*}(f^{*}\Omega_{X}\otimes\omega_{q})[-1]\to \BL_{\sK/\fM_{g,n}}[-1]\to g^{*}\BL_{h_{1}}.
\end{equation} 
Let $\F$ be the cone of (\ref{eq:RelativeToLog}). Then $F$ induces a perfect obstruction theory of $\sK$ relative to $S^{ext}$. By \cite[A1]{Bryan-Leung}, the virtual fundamental class associated to $F$ coincide with the Behrend-Fantechi class. 

By (\ref{eq:CotLogCurve}), $F$ is the cone of 
\[q_{*}(f^{*}\Omega_{X}\otimes\omega_{q})[-1]\to q_{*}(\Omega_{C_{g,n}/\fM_{g,n}}^{log}\otimes\omega_{q})[-1].\]
Comparing with (\ref{equ:obs1}), we have $F\cong E_{\sK/S}$, where $E_{\sK/S}$ is the perfect obstruction theory constructed in Section \ref{ss:ConstructVir}. This gives the following result:

\begin{prop}
If the target is a smooth projective variety over $\C$, the virtual fundamental class $[\sK]$ induced by $E_{\sK/S}$ coincides with the Behrend-Fantechi class.
\end{prop}

\subsection{The perfect obstruction theory under base change}
Assume that the stack $\sK_{\sB}$ behaves well under base change as in Remark \ref{rem:BaseChange}. Then the stack $\sK_{\sB}^{log}$ is also well-behaved under base change. Consider a morphism $\psi:T \to \sB$. We have the following base change diagram from Remark \ref{rem:BaseChange}:
\[
\xymatrix{
\sK^{log}_{T} \ar[r]^{\phi} \ar[d] & \sK_{\sB}^{log} \ar[d] \\
T^{tw} \ar[r] \ar[d] & \sB^{tw} \ar[d] \\
T \ar[r]^{\psi} & \sB.
}
\]

\begin{prop}\label{prop:PerObsBaseChange}
\[\E_{\sK_{T}^{log}/T^{tw}}\cong \mbox{L}\phi^{*}\E_{\sK_{\sB}^{log}/\sB^{tw}}.\]
\end{prop}
\begin{proof}
Note that we have the following catesian diagram of log stacks:
\[
\xymatrix{
\sC_{T} \ar[r]^{\phi'} \ar[d]_{q'} & \sC \ar[d]^{q}\\
\sK^{log}_{T} \ar[r]^{\phi}  & \sK_{\sB}^{log} ,
}
\]
where $\sC$ is the universal curve over $\sK_{\sB}^{log}$. Denote by $\BL_{\square}'$ the corresponding complex over $\sC_{T}$ as in Section \ref{ss:ConstructVir}. Since $\BL_{\square}$ is the cone of (\ref{eq:ConeOfCot}), and the family of targets are both log smooth and integral over the base, we can pull-back the complex via $\psi$. Then we have a canonical isomorphism 
\[\BL_{\square}'\cong \mbox{L}(\phi')^{*}\BL_{\square}.\] 
Furthermore, we have the canonical isomorphism of dualizing complexes 
\[\omega_{q'}\cong\mbox{L}(\phi')^{*}\omega_{q}.\] 
Thus, we have the following:
\[\E_{\sK_{T}^{log}/T^{tw}}\cong\mbox{R}q_{*}'(\mbox{L}(\phi')^{*}(\BL_{\square}'\otimes\omega_{q'}[-1]))\cong \mbox{L}\phi^{*}(\mbox{R}q_{*}(\BL_{\square}\otimes\omega_{q}[-1]))\cong \mbox{L}\phi^{*}\E_{\sK_{\sB}^{log}/\sB^{tw}}.\]
Note that the middle isomorphism follows from the usual base change theorem and the proof of \cite[1.3]{Bondal-Orlov}, where the proof still works when the fiber is a family of prestable curves. This finishes the proof of the statement. 
\end{proof}

Denote by $[\sK_{T}^{log}]$ the virtual fundamental class of $\sK_{T}^{log}$ relative to $T^{etw}$ given by $\E_{\sK_{T}^{log}/T^{tw}}$. Then the above proposition implies that:
\begin{cor}
$[\sK_{T}^{log}] = \phi^{*}([\sK_{\sB}^{log}])$.
\end{cor}

\section{Expanded pairs and degenerations}\label{sec:target}
Following \cite{Jun1} and \cite{AF}, we will introduce expanded pairs, degenerations, and their stacks. Since they are log FM type spaces as in Definition \ref{def:FM-space}, we will consider the corresponding log stacks as in Lemma \ref{lem:target-stack-alg}.

\subsection{Log twisted half accordions and log twisted expanded pairs}\label{ss:ExpPair}
Consider a smooth pair $(X,D)$, where $X$ is a smooth projective scheme, and $D\subset X$ is a connected smooth divisor. Denote by $N_{D/X}$ the normal bundle of $D$ in $X$, and $\Proj=\Proj_{D}(\sO_{D}\oplus N_{D/X})$ the projective completion of the normal bundle $N_{D/X}$. Note that we have a canonical isomorphism $\Proj_{D}(\sO_{D}\oplus N_{D/X})\cong \Proj_{D}(\sO_{D}\oplus N_{D/X}^{\vee})$. Let $D^{+}$ and $D^{-}$ be the divisors in $\Proj$ with normal bundles $N_{D/X}$ and $N_{D/X}^{\vee}$ respectively. Then we have canonical isomorphisms of divisors 
\[D\cong D^{-}\cong D^{+}.\] 

Consider $\Proj_{i}$ for $i=0,\cdots,n-1$, which are $n$ copies of $\Proj$. Using the canonical isomorphisms of divisors, we can glue $X$ and $\Proj_{1}$ along $D$ and $D^{-}$, write the resulting normal crossing singularity to be $D_{1}$; and glue $\Proj_{i}$ and $\Proj_{i+1}$ along $D^{+}\subset \Proj_{i}$ and $D^{-}\subset \Proj_{i+1}$, write the resulting normal crossing singularity to be $D_{i+1}$. In such a way, we obtain a scheme
\begin{equation}\label{eq:HalfAccordion}
X[n]=X\coprod_{D_{1}}\Proj_{1}\coprod_{D_{2}}\cdots\coprod_{D_{n-1}}\Proj_{n-1}.
\end{equation}
Note that there is a natural projection $\Proj \to D$, which induces a projection $X[n]\to X$ by contracting $\Proj_{i}$ for all $i$. Denote by $D_{n}$ the divisor $D^{+}$ in $\Proj_{n-1}$. Now we have a sequence of morphisms $D\cong D_{n}\hookrightarrow X[n] \to X$, we call it the {\em $n$-th half accordion over $(X,D)$}. 

Denote by $\sT^{u}$ the stack of {\em expanded pairs over $(X,D)$}, which associate to every reduced scheme $S$ a sequence of morphisms $D\times S\hookrightarrow \sX \to X\times S$ over $S$, such that
\begin{enumerate}
 \item $D\times S \hookrightarrow \sX$ is a closed embedding;
 \item the family $\sX \to S$ is flat;
 \item for every point $s\in S$, the fiber $D\to \sX_{s}\to X$ is a half accordion over $(X,D)$.
\end{enumerate} 

\begin{rem}
The stack $\sT^{u}$ is studied in \cite{ACFW} and \cite{AF}, which has the following properties.
\begin{enumerate}
 \item The stack $\sT^{u}$ is a connected, smooth algebraic stack of dimension $0$.
 \item The stack do not depend on the choice of the pair $(X,D)$.
\end{enumerate}
\end{rem}

\begin{rem}\label{rem:CanLogForHalfAccd}
By the above gluing construction, any expanded pair $D\times S \to \sX \to X\times S$ over a reduced scheme $S$ is d'semistable along its singularity as in \cite{LogSS}. Thus the family $\sX\to S$ has a canonical log structures as described in Remark \ref{rem:log-special}. Therefore, the family $\sX\to S$ is a log FM type space with the divisor $D\times S \hookrightarrow \sX$.  
\end{rem}

Now we introduce the log twisting for expanded pairs.

\begin{defn}\label{def:LogTwHalfAccod}
A {\em (un)extended log twisted expanded pair over $S$} is a (un)extended log twisted smooth pair $(D\times S \to \sX \to X\times S, \sM_{S}^{\sX/S}\to\sM_{S})$ as in Definition \ref{def:LogTwistedFMPair}, where $D\times S \to \sX \to X\times S$ is a family of expanded pairs over $S$. In the rest of the paper, we use $(\sX\to S,\sM_{S})$ for the log twisted half accordion over $S$, when no confusion could arise.
\end{defn}

Same as in Section \ref{ss:LogStrTarget}, we have smooth algebraic stacks $\sT^{tw}$ and $\sT^{etw}$ parameterizing unextended and extended log twisted expanded pairs respectively. Denote by $\sX^{u}$, $\sX^{tw}$, and $\sX^{etw}$ the universal family of $\sT^{u}$, $\sT^{tw}$, and $\sT^{etw}$ respectively. We have the canonical log structures $\sM_{\sX^{u}}$, $\sM_{\sX^{tw}}$, $\sM_{\sX^{etw}}$, and $\sM_{\sT^{u}}$, $\sM_{\sT^{tw}}$, $\sM_{\sT^{etw}}$ on the families and bases respectively.

\begin{rem}\label{rem:ConstructStackPair}
We would like to describe the structure of $\sT^{u}$ given in \cite{ACFW} and \cite{AF}. Denote by $U_{k}=[\A^{k}/\G_{m}^{k}]\cong \sA^{k}$, where each $\G_{m}$ acts on the corresponding copy of $\A$ by multiplication, and $k$ is a positive integer. This $U_{k}$ carries a universal family $\sX_{k}$ of expanded degenerations with splitting divisors labeled by elements of $\{1,\cdots,k\}$. The locus $\Delta_{i}$ where the $l$-th splitting divisor persists corresponds to $x_{i}=0$ in $\A^{k}$. Thus we see that the canonical log structure $\sM_{U_{k}}^{\sX_{k}/U_{k}}$ corresponds to the smooth divisors $\Delta_{i}$ for $i=1,\cdots,k$. For each strictly increasing map $\{1,\cdots,k\}\to \{1,\cdots,k'\}$, there is a natural open embedding 
\begin{equation}\label{equ:expand-pair}
U_{k}\cong U_{k}\times [\A^{*}/\G_{m}]^{k'-k}\to U_{k'}.
\end{equation}
The set of all such $U_{k}$ forms an \'etale open cover of $\sT^{u}$.

The stack $\sT^{u}$ is given by taking the limit of $U_{k}$ in the categorical sense with the arrow described in (\ref{equ:expand-pair}).
\end{rem}

\begin{prop}\label{prop:LogTwistPair}
The stack $\sT^{tw}$ is identical to the stack $\sT^{tw}_{1}$ in \cite{AF} which parameterizing twisted expanded pairs with no twists along the smooth divisor.
\end{prop}
\begin{proof}
Denote $\Br=(r_{1},\cdots,r_{k})$ to be a tuple of $k$ positive integers. Now consider $U_{k,\Br}$, which is given by taking the $r_{i}$-th root stack along $\Delta_{i}$ in $U_{k}$ as above, see Section \ref{ss:RootLog} for taking root stacks. Then by \cite{LogStack} and Definition \ref{def:LogTwHalfAccod}, such $U_{k,\Br}$ has a universal family $\sX_{k,\Br}$ of log twisted half accordions with the log twisting index $r_{i}$ along the splitting divisor in $\sW_{k}$ that corresponding to $\Delta_{i}$. Similarly, we have a map $U_{k,\Br}\to U_{k',\Br'}$ induced by the increasing map $\phi: \{1,\cdots, k\}\to \{1,\cdots,k'\}$ with the condition $r'_{\phi(i)}=r_{i}$. Clearly $\sT^{tw}$ is the categorical limit of such $U_{k,\Br}$ with respect to the maps above. This proves what we want. 
\end{proof}

\subsection{Log twisted expanded degeneration}\label{ss:ExpandedDeg}
Consider $\pi: W \to B$, a flat, projective family of schemes over a smooth curve $B$. Let $0\in B$ be a closed point such that $\pi$ is smooth away from $W_{0}=W \times_{B}0$, and the central fiber $W_{0}$ is reducible with two smooth components $X_{1}$ and $X_{2}$ intersect transversally along the smooth divisor $D\subset W_{0}$. We view $D$ as a smooth divisor in $X_{i}$, and write $D_{i}\subset X_{i}$. So we have two smooth pairs $(X_{1},D_{1})$ and $(X_{2},D_{2})$. Denote by $N_{i}$ the normal bundle of $D_{i}$ in $X_{i}$ for $i=1,2$. Since $W_{0}$ is smoothable, we have $N_{1}\cong N_{2}^{\vee}$. 

Same as in the case of half accordions, we can glue $X_{1}[n]$ and $X_{2}$ along $D_{n}$ and $D_{2}$. Recall that $D_{n}$ is the divisor corresponding to the normal bundle $N_{1}$ in $\Proj_{n-1}\cong\Proj_{D}(\sO_{D}\oplus N_{1})$. We still use $D_{n}$ to denote the image of $D_{2}$ in the resulting gluing $W_{0}[n]$. We call $W_{0}[n]$ the {\em $n$-th accordion over $W_{0}$}. Note that we have a projection $W_{0}[n]\to W_{0}$ by contracting all $\Proj$ to the divisor $D$. We have the following definition from \cite{ACFW} and \cite{AF}:

\begin{defn}
A family of projective morphism $\sW\to S\times_{B} W$ over a reduced $B$-scheme $S$ is called an {\em expanded degeneration of $\pi$ over $S$} if:
\begin{enumerate}
 \item the morphism $\sW \to S$ is flat and proper;
 \item for every point $s\in S$ the fiber $\sW_{s}\to W_{p(s)}$ is an isomorphism if $p(s)\neq 0$, and is an accordion over $W_{0}$ if $p(s)=0$, where $p: S \to B$ is the structure morphism.
\end{enumerate}
Sometimes we use $\sW\to S$ to denote the expanded degeneration when there no confusion about the contraction. 
\end{defn}

\begin{rem}
Same as in Remark \ref{rem:CanLogForHalfAccd}, any family of expanded degenerations $\sW \to S$ give a family of log FM type spaces. 
\end{rem}

Following Section \ref{ss:LogStrTarget}, we have:

\begin{defn}
A {\em (un)extended log twisted expanded degeneration over $S$} is a (un)extended log twisted FM type space $(\sW \to S, \sM_{S}^{\sW/S}\to\sM_{S})$ as in Definition \ref{def:log-twist-FM}, where $\sW \to S$ is a family of expanded degeneration of $\pi$ over $S$. Denote by $\Br=(r_{1},\cdots,r_{m})$ the log twisting index along the corresponding singular locus of $\sW$. 
\end{defn}

\begin{rem}\label{rem:ConstructStackDeg}
Let $\fT^{u}$ be the stack parameterizing expanded degenerations of $\pi$. By \cite{ACFW} and \cite{AF}, it is a smooth algebraic stack which has the following description. 

For any positive integers $k$, consider the product morphism $\sA^{k}\to \sA$, and the morphism $B\to \sA$ given by the divisor $0\in B$. Denote by $\tilde{U}_{k}=\sA^{k}\times_{\sA}B$. It has a universal family of expanded degenerations of $\pi$, with the splitting singular divisors labelled by the set $\{1,\cdots,k\}$ such that the locus $\Delta_{l}$ where the $l$-th singular locus persists corresponds to $x_{l}=0$ in $\A^{k}$. Same as in Remark \ref{rem:ConstructStackPair}, the stack $\fT$ is the limit of such $\tilde{U}_{k}$ given by the map $\tilde{U}_{k}\to \tilde{U}_{k'}$ for any strictly increasing map $\{1,\cdots,k\}\to \{1,\cdots,k'\}$. The canonical log structure on the stack is given by the divisors $\Delta_{l}$ for all $l$.
\end{rem}

Denote by $\fT^{tw}$ and $\fT^{etw}$ the stack parameterizing log twisted and extended log twisted expanded degenerations of $\pi$. We have the following:

\begin{prop}\label{prop:TwExpansion}
\begin{enumerate}
\item The stack $\fT^{tw}$ is identical to the stack in \cite{AF} which parameterizing twisted expanded degeneration.
\item The stacks $\fT^{u}$ and $\fT^{tw}$ only depend on the base $(B,0)$.
\item The stacks $\fT^{u}$ and $\fT^{tw}$ are smooth, connected algebraic stacks of dimension $1$.
\item There is a natural open embedding $\fT^{u} \to \fT^{tw}$ and a natural right inverse $\fT^{tw}\to \fT$.
\end{enumerate}
\end{prop}
\begin{proof}
The proof of the first statement is identical to Proposition \ref{prop:LogTwistPair} by taking the root stack along divisors $\Delta_{l}$. The rest follows from \cite{ACFW}. 
\end{proof}

For later use, denote by $\sW^{etw}$, $\sW^{tw}$, and $\sW^{u}$ the universal family over $\fT^{etw}$, $\fT^{tw}$, and $\fT^{u}$ respectively.
 
\subsection{Splitting expanded degeneration}\label{ss:SplitExpDeg}
We have seen in Proposition \ref{prop:LogTwistPair} and \ref{prop:TwExpansion} that the stacks parameterizing log twisted expanded pairs and degenerations coincide with the stacks with the stack twisting in \cite{ACFW} and \cite{AF}. Next, we would like to gathering several results from \cite{ACFW} and \cite{AF} about the splitting of those stacks. However, we will state the results for our log twisted stacks. These results will be used later for the proof of degeneration formula. 

Denote by $\fT^{u}_{0} = \fT^{u}\times_{B}0$ and $\fT^{tw}_{0}=\fT^{tw}\times_{B}0$. It is clear that $\fT^{u}_{0}$ is a reduced normal crossing divisor in $\fT^{u}$, its inverse image in $\fT^{tw}$ is the non-reduced stack $\fT^{tw}_{0}=\sum_{r}r\fT^{r}_{0}$, where $\fT^{r}_{0}$ is the divisor corresponding to expansions having a splitting divisor with log twisting index $r$. Then we have the following:

\begin{prop}\cite[2.3.1]{AF}
\begin{enumerate}
 \item The stacks $\fT^{u}_{0}$ and $\fT^{tw}_{0}$ parametrize untwisted and log twisted expansions of the singular fiber $W_{0}$ respectively.
 \item These two stacks are independent of $W_{0}$.
\end{enumerate}
\end{prop} 

Denote by $\fT^{r,spl}_{0}$ the stack of log twisted accordions with a choice of a splitting divisor $D_{i}$ with log twisting index $r$. We can similarly define $\fT^{u,spl}_{0}$ in the untwisted case. We have a natural map $\fT^{r,spl}_{0}\to \fT^{tw}_{0}$ and $\fT^{u,spl}\to\fT^{u}_{0}$. 

Consider the stack $\fQ=\fT^{u,spl}_{0}\times_{\fT^{u}_{0}}\fT^{tw}_{0}$. As shown in \cite{AF}, it has a decomposition of disjoint union 
\begin{equation}\label{equ:split-target}
\fQ=\coprod_{r} \fQ_{r},
\end{equation} 
where over the reduction of $\fQ_{r}$, the splitting divisor in the universal family has log twisting index $r$.

\begin{lem}\label{lem:RedSplitting}\cite[2.4.1]{AF} 
The morphism $\fT^{r,spl}_{0}\to \fQ_{r}$ is of degree $1/r$.
\end{lem}

Consider the universal family $\sW^{r,spl}$ over $\fT^{r,spl}_{0}$. If we normalize along the splitting divisor, we have two family of log twisted expanded pairs. Conversely, we have:
\begin{lem}\label{lem:GlueHalfAcc}\cite[2.4.2]{AF}
The natural morphism $\fT^{r,spl}_{0}\to \sT^{tw}\times \sT^{tw}$ corresponding to the two components of the partial normalization of the universal family is a gerbe banded by $\mu_{r}$; in particular it has degree $1/r$.
\end{lem}
\begin{proof}
We give a proof in our log case. If we glue the resulting two components, locally along the splitting divisor $D$, it has a canonical log structure given by $\N^{2}\to \sM_{\sW}$, where we denote $a$ and $b$ to be the two standard generators of $\N^{2}$. In order to put log twisting, we take the $r$-th root $e$ of $a+b$. Then the new log structure is generated by $e,a$ and $b$ with the only relation $re=a+b$. Clearly, there is a $\mu_{r}$ action on the new log structure by the multiplication on $e$ but without changing $a$ and $b$. 
\end{proof}

\begin{rem}\label{rem:LineBundleGlue}
By the above lemma, we have the canonical isomorphism $\fT^{u,spl}_{0}\cong \sT^{tw}\times\sT^{tw}$. In fact the universal family $\sW^{u,spl}$ over $\fT^{u,spl}_{0}$ can be obtained by gluing the two universal families over each copy of $\sT^{u}$, and then put the standard log structure along the splitting divisor.

Now we compare $\fT^{r,spl}_{0}$ and $\fT^{u,spl}_{0}$. As explained in \cite{AF} that $\fT^{u,spl}_{0}$ is a smooth divisor in $\fT^{tw}$. We use the notations in Remark \ref{rem:ConstructStackDeg}. Consider $\DD=(\Delta_{1},\cdots,\Delta_{l})$, and $\Br=({r_{1},\cdots,r_{l}})$. By Proposition \ref{prop:TwExpansion}(1) and Section \ref{sss:LogRoot}, we may assume that locally we have a chart $(\tilde{U}_{k})_{\DD,\Br}$ for $\fT^{tw}$. Then $\fT^{u,spl}_{0}$ has a chart given by the image of $\Delta_{l+1}$ in $(\tilde{U}_{k})_{\DD,\Br}$. Here, we assume that the splitting divisor is indexed by $l+1$. 

Still use $\Delta_{l+1}$ to denote this divisor in $U:=(\tilde{U}_{k})_{\DD,\Br}$. Consider the root stack $U_{\Delta_{l+1},r}$, and the special divisor $\sD$ which is the $r$-th root of $\Delta_{l+1}$ as in subsection \ref{ss:RootLog}. Then by the definition of $\fT^{r,spl}_{0}$, we have a chart $U_{\Delta_{l+1},r}$ for $\fT^{tw}$, in which the stack $\fT^{r,spl}_{0}$ is given by the divisor $\sD$. Locally the map $\fT^{r,spl}_{0}\to\fT^{u,spl}_{0}$ is induced by $U_{\Delta_{l+1},r}\to U$.
\end{rem}

\section{Relative and Degenerate Gromov-Witten invariants}\label{sec:GW}
In this section, we introduce the Gromov-Witten invariants with the target given by the extended log twisted expanded pairs or degenerations.

\subsection{Gromov-Witten invariants for expanded degenerations}\label{ss:GWDeg}
We use the notations as in Section \ref{ss:ExpandedDeg}. Let $\beta$ be a curve class in the fiber of $\pi:W\to B$. Denote by $g\geq 0$ the genus, and $n$ the number of marked points. Consider the singular fiber $W_{0}$, and the stack $\fT^{etw}_{0}$ parameterizing the expanded degenerations of $W_{0}$ with extended log twisting. Denote by $\sK^{log}_{g,n}(W_{0},\beta)$ the stack parameterizing log stable maps to the extended log twisted expansions of $W_{0}$. Let $\sK_{g,n}(W_{0},\beta)$ be the stack parameterizing the corresponding stable predeformable maps as in Remark \ref{rem:predeformable}. It was shown in \cite{Jun1} and \cite{AF} that $\sK_{g,n}(W_{0},\beta)$ is a proper DM-stack. By Theorem \ref{thm:LogStackForMap}, the stack $\sK^{log}_{g,n}(W_{0},\beta)$ is also a proper DM-stack.

Denote by $\sC$ the universal curve over $\sK^{log}_{g,n}(W_{0},\beta)$, and $\{\Sigma_{i}\}_{i=1}^{n}$ the $n$ universal sections. Note that we have a projection $\sW_{0}^{etw} \to W$ from the universal family of expansions. Consider the following universal diagram:
\[
\xymatrix{
\sC \ar[r]^{f} \ar[d] & \sW^{etw}_{0} \ar[rr] && W_{0} \\
\sK^{log}_{g,n}(W_{0},\beta).
}
\]
The composition of the top arrows gives a pre-stable map. We take its stabilization:
\begin{equation}\label{diag:contract-family}
\xymatrix{
C \ar[r]^{f} \ar[d] &  W_{0} \\
\sK^{log}_{g,n}(W_{0},\beta).
}
\end{equation}
The section $\Sigma_{i}$ induces an evaluation morphism of usual stacks:
\[\Bev_{i}: \sK^{log}_{g,n}(W_{0},\beta) \to W_{0}.\]
Consider the $i$-th descendant class of the stabilized curve:
\[\psi_{i}=c_{1}(\Sigma^{*}_{i}\omega_{C/\sK^{log}_{g,n}(W_{0},\beta)}).\]

\begin{defn}
Given positive integer $m_{i}$ and cohomology class $\gamma_{i}\in H^{*}(W_{0})$ for each $i=1,\cdots, n$, we define the Gromov-Witten invariant with gravitational descendants by:
\[\left\langle \prod_{i=1}^{n}{\tau_{m_{i}}(\gamma_{i})}\right\rangle_{g,n,\beta}^{W_{0}}:= \mbox{deg}\left(\prod_{i=1}^{n}{(\psi_{i}^{m_{i}})\cdot\Bev_{i}(\gamma_{i})}\cap [\sK^{log}_{g,n}(W_{0},\beta)]\right),\]
where the virtual fundamental class $[\sK^{log}_{g,n}(W_{0},\beta)]$ is constructed in Section \ref{ss:ConstructVir}.
\end{defn}

\begin{rem}
Note that the stack $\sK^{log}_{g,n}(W_{0},\beta)$ and the Gromov-Witten invariant does not depent on a global smoothing $W\to B$ of $W_{0}$. We only need that $W_{0}$ to be a log FM type space over a point. 
\end{rem}

Similarly, we have a stack $\sK_{g,n}(W/B,\beta)$ parameterizing stable predeformable maps to the expanded degeneration of $W\to B$. Denote by $\sK_{g,n}^{log}(W/B,\beta)$ the stack parameterizing log stable maps to the extended log twisted expanded degeneration of $W\to B$. It was shown in \cite{Jun1} and \cite{AF} that $\sK_{g,n}(W/B,\beta)$ is a proper DM-stack. Furthermore it behaves well under base change. Therefore, by Remark \ref{rem:BaseChange} the stack $\sK_{g,n}^{log}(W/B,\beta)$ is also well-behaved under base change. It is not hard to see that 
\[
\sK_{g,n}^{log}(W_{0},\beta)=\sK_{g,n}^{log}(W/B,\beta)\times_{B}0.
\]

\begin{prop}
The Gromov-Witten invariants of $W_{b}=W\times_{B}b$ is independent of $b$.
\end{prop}
\begin{proof}
Since the perfect obstruction theory for $\sK_{g,n}^{log}(W/B,\beta)$ relative to $\fT^{tw}$ restrict to that of $\sK_{g,n}^{log}(W_{0},\beta)$, the statement follows from \cite[7.2(2)]{Behrend-Fantechi}. 
\end{proof}

\subsection{Relative Gromov-Witten invariants}
We fix the smooth pair $(X,D)$ as in the beginning of Section \ref{ss:ExpPair}, and the data $\Xi$ as in Definition \ref{def:WeightGraph}. Denote by $\sK_{\Xi}^{log}(X,D)$ the stack parameterizing $\Xi$-relative log stable maps to the extended log twisted expanded pairs. Let $\sK_{\Xi}(X,D)$ be the stack parameterizing the corresponding underlying stable predeformable maps. Again by \cite{Jun1} and \cite{AF}, this is a proper DM-stacks. Hence $\sK_{\xi}^{log}(X,D)$ is also a proper DM-stack. Similarly, we can define evaluation maps $\Bev_{i}: \sK_{\Xi}^{log}(X,D)\to X$ of the underlying stacks, and the descendant classes $\psi_{i}$ for $i\in N$ as in Section \ref{ss:GWDeg}. 

\begin{defn}
Given nonnegative integers $m_{i}$ for $i\in L(\Xi)$, cohomology classes $\gamma_{i}\in H^{*}(X)$ for $i\in l(\Xi)$, and $\gamma_{j}\in H^{*}(D)$ for $j\in R(\Xi)$, we define relative Gromov-Witten invariants with gravitational descendants by:
\[\left\langle \prod_{i\in L(\Xi)}{\tau_{m_{i}}(\gamma_{i})}|\prod_{_{j\in R(\Xi)}}\gamma_{j}\right\rangle_{\Xi}^{(X,D)}:= deg\left(\left(\prod_{i\in L(\Xi)}{\psi_{i}^{m_{i}}\cdot\Bev_{i}^{*}\gamma_{i}}\right)\cdot \left(\prod_{j\in R(\Xi)}{\Bev_{j}^{*}\gamma_{j}}\right)\cap [\sK_{\Xi}^{log}(X,D)]\right).\]
\end{defn}

\begin{rem}
Note that the evaluation map $\Bev_{i}$ defined in this section is just usual morphism without log structures.
\end{rem}

\section{Degeneration formula}\label{sec:degeneration-formula}

\subsection{Statement of the degeneration formula}
Following the method in \cite{AF}, we prove the degeneration formula in the logarithmic setting. We first fix the following notations.

Consider the log FM type space $W_{0}=X_{1}\coprod_{D}X_{2}$ and two smooth pairs $(X_{1},D_{1})$ and $(X_{2},D_{2})$ as in Section \ref{ss:ExpandedDeg}. Denote by $H$, $H_{1}$ and $H_{2}$ the monoids of curve classes in $W_{0}$, $X_{1}$, and $X_{2}$ respectively. Fix the data $\Gamma=(g,N,\beta)$, where $g\geq 0$ is the genus, $N=\{1,\cdots,n\}$ is the index set for marked points, and $\beta\in H$ is the curve class. We would like to view $\Gamma$ as a weighted graph having one vertex with genus $g$, curve class $\beta$, and legs labelled by $N$. For simplicity, we use $\sK$ to denote the stack $\sK^{log}_{g,n}(W_{0},\beta)$. 

For smooth pairs we introduce the following data:

\begin{defn}
A splitting $\eta$ of $\Gamma$ is an ordered pair $\eta=(\Xi_{1},\Xi_{2})$ where
\begin{enumerate}
 \item $\Xi_{1}$ and $\Xi_{2}$ are admissible weighted graphs as in Definition \ref{def:WeightGraph}.
 \item The labeling of legs $L(\Xi_{1})\cup L(\Xi_{2})=N$ gives a partition of $N$.
 \item The labeling of roots $R(\Xi_{1})\longleftrightarrow M\longleftrightarrow R(\Xi_{2})$ is ordered by the set $M$ disjoint from $N$.
 \item For elements $r\in R(\Xi_{1})$ and $r'\in R(\Xi_{2})$ that correspond to the same element $j\in M$, we associate the same tangency multiplicity $c_{j}$ for both $r$ and $r'$.
 \item For each vertex $v\in V(\Xi_{i})$, we assign the genus $g(v)$ and a curve class $\beta(v)\in H_{i}$.
\end{enumerate}
The above data must satisfy the following conditions:
\begin{enumerate}
 \item By Gluing $\Xi_{1}$ and $\Xi_{2}$ along the roots labeled by $M$, we obtain a graph $\Gamma$ of genus $g$, and total weight $\beta$.
 \item For each vertex $v\in V_{\Xi_{i}}$, we have \[\sum_{j\in R_{v}}(c_{j})=(\beta(v)\cdot D_{i})_{X_{i}}, \mbox{\ \ \ for $i=1,2$.}\]
\end{enumerate}
\end{defn}

The last condition in the above definition implies that $(\beta(\Xi_{1})\cdot D_{1})_{X_{1}}=(\beta(\Xi_{2})\cdot D_{2})_{X_{2}}$, where $\beta(\Xi_{i})$ is the total weight of $\Xi_{i}$. For each $\eta$, there is a special number \[r(\eta)=l.c.m(c_{j})_{j\in M},\] which will be the log twisting index along the splitting divisor as in Remark \ref{rem:LogAtSing}. We call it the {\em log twisting index of $\eta$}.

Denote by $\Omega(\Gamma)$ the set of all splitting classes of $\Gamma$. Note that we have the symmetric group $S(M)$ acting on $\Omega$ by its action on the index set $M$ of the roots. 

\begin{defn}
Two splittings are said to be equivalent if they belong to the same $S(M)$-orbit. Denote by $\bar{\Omega}$ the set of equivalence classes, and $\bar{\eta}$ the equivalence class of $\eta\in \Omega$.
\end{defn}

Denote by $F$ a homogeneous basis of $H^{*}(D)$. For each $\delta \in F$, denote $\delta^{\vee}$ to be the dual element in the dual basis with respect to the pairing $\int_{D}{\delta\cdot\delta^{\vee}}=1$. 

\begin{thm}\label{thm:main}
For any non-negative integers $m_{i}$, and cohomology classes $\gamma_{i}\in H^{*}(W_{0})$, where $i\in N_{1}\cup N_{2}$, the following degeneration formula holds:
\begin{align*}
\left\langle \prod_{i=1}^{n}{\tau_{m_{i}}(\gamma_{i})}\right\rangle_{g,n,\beta}^{W_{0}} &= \sum_{\eta\in \Omega}\frac{\prod_{j\in M}c_{j}}{|M|!}\sum_{\delta_{j}\in F}(-1)^{\epsilon}\left\langle\prod_{i\in N_{1}}\tau_{m_{i}}(\gamma_{i})|\prod_{j\in M}\delta_{j}\right\rangle_{\Xi_{1}}^{X_{1},D}\\
& \cdot \left\langle\prod_{i\in N_{2}}\tau_{m_{i}}(\gamma_{i})|\prod_{j\in M}\delta_{j}^{\vee}\right\rangle_{\Xi_{2}}^{X_{2},D}
\end{align*}
where the sign $(-1)^{\epsilon}$ satisfies the following:
\[\prod_{i\in N}\gamma_{i}\cdot \prod_{j\in M}\delta_{j}\delta_{j}^{\vee}=(-1)^{\epsilon}\prod_{i\in N_{1}}\gamma_{i}\prod_{i\in M}\delta_{i}\prod_{i\in N_{2}}\delta_{i}\prod_{i\in M}\delta_{i}^{\vee}.\]
\end{thm}

\begin{rem}
The formation of the above degeneration formula in the log setting is identical to the one in \cite{Jun2} and \cite{AF}. 
\end{rem}

\subsection{Splitting the coarse target}
Consider the following cartesian diagram of log stacks:
\[
\xymatrix{
\sK_{\fQ} \ar[d]_{s} \ar[r] &\fQ^{ext} \ar[d] \ar[r] &\fQ \ar[d] \ar[r] &\fT^{u,spl}_{0} \ar[d] \\
\sK  \ar[r] &\fT^{etw}_{0} \ar[r] &\fT^{tw}  \ar[r] &\fT^{u}_{0}.
}
\]
Here the sup-script $ext$ means the stack parameterizing extended log structure, which is the same as in the case of curves and targets with extended log structure. Note that since the arrows are all integral log morphisms, the underlying of the above log cartesian diagram coincides with the usual cartesian diagram by removing all log structures. Identical to the case in \cite{AF}, we have the normalization map $\fT^{u,spl}_{0}\to \fT^{u}_{0}$ of pure degree $1$ in the sense of \cite[Section 5]{Co}. Then the map $\fQ^{ext}\to \fT^{etw}_{0}$ is also of pure degree $1$.

\begin{lem}\label{lem:SplitCoarse}
\[s_{*}[\sK_{\fQ}]=[\sK].\]
\end{lem}
\begin{proof}
This follows from \cite[Theorem 5.0.1]{Co}, and our construction of virtual fundamental classes which is well-behaved under base change. 
\end{proof}

\subsection{Splitting the stack target}
As in (\ref{equ:split-target}), we have the following decomposition:
\[\fQ = \coprod_{r\geq 1}{\fQ_{r}}\]
which implies the decomposition
\[\sK_{\fQ}=\coprod_{r\geq 1}{\sK_{\fQ_{r}}}.\]

The stack $\fQ_{r}$ is non-reduced, and the reduced one is $\fT_{0}^{r,spl}$, which is the stack parameterizing log twisted accordions with splitting divisor of index $r$. By Lemma \ref{lem:RedSplitting}, we have a degree $1/r$ map $\fT_{0}^{r,spl}\to \fQ_{r}$. The same proof as above gives:

\begin{lem}\label{lem:SplitStack}
Consider the (log) cartesian diagram 
\[
\xymatrix{
\sK^{spl}_{r} \ar[r] \ar[d]_{t_{r}} & (\fT_{0}^{r,spl})^{ext} \ar[d] \ar[r] &\fT_{0}^{r,spl} \ar[d]  \\
\sK_{\fQ_{r}} \ar[r] &(\fQ_{r})^{ext} \ar[r] & \fQ_{r} .
}
\]
Then \[[\sK_{\fQ_{r}}]=r \cdot (t_{r})_{*} [\sK_{r}^{spl}]. \]
\end{lem}

\subsection{Decomposing the moduli space with split target}
Given the log twisting index $r$, denote by $\bar{\Omega}_{r}$ the set of isomorphism classes of type $\eta$ with $r(\eta)=r$.
We have the refined decomposition 
\[\sK^{spl}_{r}=\coprod_{\bar{\eta}\in \Bar{\Omega}_{r}}\sK_{\Bar{\eta}}.\]
Denote by $t_{\bar{\eta}}:\sK_{\Bar{\eta}}\to \sK_{\fQ_{r}}$ the restriction of $t_{r}$. Then we have
\[[\sK_{\fQ_{r}}]=r\cdot (t_{\Bar{\eta}})_{*}\sum_{\Bar{\eta}\in\Bar{\Omega}_{r}}{[\sK_{\Bar{\eta}}]}.\]
Note that $\sK_{\eta} \to \sK_{\Bar{\eta}}$ is an $S(M)$-bundle. Therefore $\sK_{\eta}$ has an associated perfect obstruction theory and virtual fundamental class. Denote by $t_{\eta}$ the following composition 
\[\sK_{\eta} \to \sK_{\Bar{\eta}}\to \sK_{\fQ_{r}}.\]

Putting Lemmas \ref{lem:SplitCoarse} and \ref{lem:SplitStack} together, we have:
\begin{prop}\label{prop:first-decomp-vir}
\[[\sK] = \sum_{\eta\in\Omega}\frac{r(\eta)}{|M|!}\cdot(s\circ t_{\eta})_{*}[\sK_{\eta}].\]
\end{prop}

\subsection{Gluing the target}
By Lemma \ref{lem:GlueHalfAcc} the morphism $\fT^{r,spl}_{0}\to\sT_{r}^{tw}\times\sT_{r}^{tw}$ gives a gerbe banded by $\mu_{r}$. Denote by $\sK_{\Xi_{i}}=\sK_{\Xi_{i}}(X_{i}, D_{i})$ the stack of relative log stable maps with data $\Xi_{i}$ for $i=1,2$. Next we try to decompose $[\sK_{\eta}]$ in terms of $\sK_{\Xi_{1}}$ and $\sK_{\Xi_{2}}$. Consider the following catesian diagram of log stacks:
\begin{equation}\label{diag:GlueTarget}
\xymatrix{
\sK_{1,2} \ar[r] \ar[d]_{u_{\eta}} & {*} \ar[r]  \ar[d] & (\fT^{r,spl}_{0})^{ext} \ar[r] \ar[d]& \fT^{r,spl}_{0}  \ar[d] \\
\sK_{\Xi_{1}}\times\sK_{\Xi_{2}} \ar[r] & \sT^{etw}\times\sT^{etw} \ar[r] & (\sT^{tw}\times\sT^{tw})^{ext} \ar[r]& \sT^{tw}\times\sT^{tw} 
}
\end{equation}
Note that for $\sT^{etw}\times\sT^{etw}$ we have two extended log structures from each $\sT^{etw}$. The arrow $\sT^{etw}\times\sT^{etw} \to (\sT^{tw}\times\sT^{tw})^{ext}$ at the bottom is given by forgetting the order of extended log structure. 

\begin{lem}\label{lem:split-target}
The map $\sT^{etw}\times\sT^{etw} \to (\sT^{tw}\times\sT^{tw})^{ext}$ is \'etale of DM-type.
\end{lem}
\begin{proof}
The statement is local on the target. Consider charts $U_{1}$ and $U_{2}$ for $\sT^{tw}$, then we form a chart $U=U_{1}\times U_{2}\times \sA^{n}$ for $(\sT^{tw}\times\sT^{tw})^{ext}$. Clearly, the targets are covered by such charts. The preimage of this chart in $\sT^{etw}\times\sT^{etw}$ is given by assign copies of $\sA$ to each $\sT^{etw}$, i.e.
\[(\sT^{etw}\times\sT^{etw})\times_{(\sT^{tw}\times\sT^{tw})^{ext}}U=\coprod_{n_{1}+n_{2}=n}(U_{1}\times\sA^{n_{1}})\times(U_{1}\times\sA^{n_{2}}).\]
The statement follows from this. 
\end{proof}

\begin{rem}
The stack $\sK_{1,2}$ parametrizes a glued log twisted target, along with a pair of relative log stable maps to the two parts of the partial normalization of the glued target with data $\Xi_{1}$ and $\Xi_{2}$.
\end{rem}

In Section \ref{ss:ConstructVir}, we construct the perfect obstruction theory $\E_{\sK_{\Xi_{i}}/\sT^{tw}}$ for $\sK_{\Xi_{i}}$ relative to $\sT^{etw}$. Thus, we have a perfect obstruction theory $\E_{\sK_{\Xi_{1}}/\sT^{tw}}\oplus\E_{\sK_{\Xi_{2}}/\sT^{tw}}$ for $\sK_{\Xi_{1}}\times\sK_{\Xi_{2}}$ relative to $\sT^{etw}\times\sT^{etw}$. By Lemma \ref{lem:GlueHalfAcc} and \ref{lem:split-target}, the vertical maps in (\ref{diag:GlueTarget}) are \'etale and of DM-type. The perfect obstruction theory $\E_{\sK_{\Xi_{1}}/\sT^{tw}}\oplus\E_{\sK_{\Xi_{2}}/\sT^{tw}}$ pulls back, and defines a perfect obstruction theory for $\sK_{1,2}/{*}$. Thus, it defines a perfect obstruction theory for $\sK_{1,2}$ relative to $(\fT^{r,spl}_{0})^{ext}$. Denote by $[\sK_{1,2}]$ and $[\sK_{\Xi_{1}}\times\sK_{\Xi_{2}}]=[\sK_{\Xi_{1}}]\times[\sK_{\Xi_{2}}]$ the associated virtual fundamental classes. Therefore, we have the following result:

\begin{lem}\label{lem:GlueTarget}
\[[\sK_{\Xi_{1}}\times\sK_{\Xi_{2}}]=r\cdot (u_{\eta})_{*}[\sK_{1,2}].\]
\end{lem}

\subsection{Gluing the underlying maps}\label{ss:GlueUnderCurve}
Denote by $f_{i}:\sC_{i}\to \sX_{i}$ for $i=1,2$ the universal relative log stable maps over $\sK_{1,2}$ with data $\Xi_{1}$ and $\Xi_{2}$ respectively, and $\sW$ the universal glued log target over $\sK_{1,2}$. Let $\sD=D\times \sK_{1,2}$ be the universal splitting divisor of $\sW$ over $\sK_{1,2}$, and $\sD_{1}$ and $\sD_{2}$ the universal divisors in $\sX_{1}$ and $\sX_{2}$ respectively. Then we have natural isomorphisms $\sD\cong\sD_{1}\cong\sD_{2}$. For convenience, we use $\sD$ to denote both $\sD_{1}$ and $\sD_{2}$ if there is no confusion. Denote by $\pi:\sD\to D$ the canonical projection.

For each $i \in M$, we have sections $\Sigma_{\Xi_{1},i}:\sK_{1,2}\to \sC_{1}$ and $\Sigma_{\Xi_{2},i}:\sK_{1,2}\to \sC_{2}$ given by the universal pre-stable curves $\sC_{1}$ and $\sC_{2}$ respectively. Consider the following compositions:
\[\Bev_{\Xi_{1},i}:=\pi\circ f_{1}\circ\Sigma_{\Xi_{1},i}: \sK_{1,2} \to D \mbox{\ \ and\ \ } \Bev_{\Xi_{2},i}:=\pi\circ f_{2}\circ\Sigma_{\Xi_{2},i}: \sK_{1,2}\to D.\]
Combine those maps, we have
\[\Bev_{1,2}=\prod_{i\in M}\Bev_{\Xi_{1},i}\times\Bev_{\Xi_{2},i}:\sK_{1,2}\to (D\times D)^{|M|}.\]
Now consider the following cartesian diagram:
\begin{equation}\label{diag:GlueUnderMap}
\xymatrix{
\sK_{1,2}' \ar[rr]^{v'} \ar[d]_{\Bev_{\eta}'} && \sK_{1,2} \ar[d]^{\Bev_{1,2}} \\
D^{|M|} \ar[rr]^{\Delta\ \ } && (D\times D)^{|M|},
}
\end{equation}
where the bottom map $\Delta$ is the diagonal morphism. 

By pulling back the universal family over $\sK_{1,2}$, there are two universal relative log stable maps $f_{i}':\sC_{i}'\to \sX_{i}'$ of type $\Xi_{i}$ for $i= 1,2$, and the glued log target $\sW'$ over $\sK_{1,2}'$. Considering only the underlying structure, the above cartesian diagram induces the following commutative diagram (without log structures):
\begin{equation}\label{diag:UndGlue}
\xymatrix{
 & \sG \ar[rr] \ar[dl] \ar@{-->}[dd] && \sC_{1}' \ar[dl]^{v_{1}} \ar[dd]^{f_{1}'} \\
\sC_{2}' \ar[rr]^{\ \ \ \ \ v_{2}} \ar[dd]_{f_{2}'} && \sC \ar[dd] & \\
 & \sD' \ar@{-->}[rr] \ar@{-->}[dl] && \sX_{1}' \ar[dl]\\
\sX_{2}' \ar[rr] && \sW' &.
}
\end{equation}
Here $\sG$ is the union of the sections numbered by $M$, and the top and bottom squares are push-out diagrams. Thus, we obtain a glued underlying curve $\sC$, and a usual stable morphism $f: \sC\to \sW'$ of underline schemes over $\sK_{1,2}'$. Furthermore, it is not hard to see that $f$ is a family of predeformable maps. Therefore, the stack $\sK_{1,2}'$ parametrizes all possible underlying maps of log stable maps with their splittings to two relative log stable maps of data $\Xi_{1}$ and $\Xi_{2}$, and the glued log targets.

\subsection{Log structures on $\sK_{1,2}'$}\label{ss:LogOnK'}
\subsubsection{The log structure from the splitting divisor}
Note that on $\sK_{1,2}'$ there is a canonical log structure $\sM^{1,spl}$ given by the smoothing of the splitting divisor in $\sW'$. Let $\sM^{\sW'/\sK_{1,2}'}$ be the log structure on $\sK_{1,2}'$ with respect to the glued unextended log target $\sW'$. Note that such log structure is obtained by pulling back the log structure from the universal family over $\fT^{r,spl}_{0}$. By Remark \ref{rem:LineBundleGlue} and Lemma \ref{lem:GlueHalfAcc}, there is a rank $1$ locally free sub-log structure $\sM^{r,spl}\subset\sM^{\sW'/\sK_{1,2}'}$, and a morphism of log structures $\sM^{1,spl}\to\sM^{r,spl}$, which locally have chart:
\[
\xymatrix{
\N \ar[d] \ar[r]^{\times r} & \N \ar[d] \\
\sM^{1,spl} \ar[r] & \sM^{r,spl}.
}
\]

Given a positive integer $c|r$, denote by $l=\frac{r}{c}$. We define a new log structure $\sM^{r,spl}_{l}$ on $\sK_{1,2}'$ as follows. Let $e$ be the local generator of $\sM^{r,spl}$. Then locally we define $\sM^{r,spl}_{l}$ to be the sub-log structure of $\sM^{r,spl}$ generated by $l\cdot e$. It is not hard to check that such $\sM^{r,spl}_{l}$ is well-defined, and is a rank $1$ locally free log structure. Since $l|r$, the morphism $\sM^{1,spl}\to\sM^{r,spl}$ induces a morphism $\sM^{1,spl}\to\sM^{r,spl}_{l}$, which locally have chart:
\[
\xymatrix{
\N \ar[d] \ar[r]^{\times c} & \N \ar[d] \\
\sM^{1,spl} \ar[r] & \sM^{r,spl}_{l}.
}
\]
By our discussion in Section \ref{ss:RootLog}, this gives a map $\sK_{1,2}'\to\fT^{c,spl}_{0}$. 

\subsubsection{The log structures from the splitting nodes}
Denote by $p_{i}\in \sC$ the universal splitting node, which splits to the two marked points numbered by $i\in M$. Note that there is a rank $1$, locally free log structure $\sN_{i}$ on $\sK'_{1,2}$, corresponding to smooth of the node $p_{i}$. Let $c_{i}$ be the contact order at $p_{i}$.

\begin{lem}\label{lem:compare-standard-log}
There is a natural map $\sM^{1,spl}\to \sN_{i}$, locally given by the following chart:
\[
\xymatrix{
\N \ar[d] \ar[r]^{\times c} & \N \ar[d] \\
\sM^{1,spl} \ar[r] & \sN_{i}.
}
\]
Therefore, we have a morphism $\sK_{1,2}'\to\fT^{c_{i},spl}_{0}$.
\end{lem}
\begin{proof}
Consider a geometric point $\bar{s}\in \sK_{1,2}'$. Then there is a geometric point $\bar{t}\in p_{i}$ over $\bar{s}$, and $\bar{t}':=f(\bar{t})\in \sW'$. Locally around $\bar{t}$, we can choose coordinates $x$ and $y$ which correspond to the coordinates of the two components intersecting at the node $p_{i}$, such that $e=\log x +\log y$ gives the local generator of $\sN_{i}$ on the base. Similarly for $\bar{t}'$, we can choose local coordinates $u$ and $v$ that correspond to the coordinates of the two components intersecting along the splitting divisor, such that $e'=\log u+\log v$ gives the local generator of $\sM^{1,spl}$ on $\sK_{1,2}'$. Since $f$ is predeformable along $p_{i}$, we may assume that $f^{*}(u)=x^{c_{i}}$ and $f^{*}(v)=y^{c_{i}}$. We can locally define a morphism $\sM^{1,spl} \to \sN_{i}$ by $e' \mapsto c_{i}\cdot e$. Note that such morphism is independent of the choice of the chart, and only relies on the underlying map $f$. Thus, we can glue the local construction and obtain a morphism $\sM^{1,spl} \to  \sN_{i}$ as in the statement. 
\end{proof}

\subsection{Comparison of stacks $\sK_{1,2}'$ and $\sK_{\eta}$}\label{ss:DefLogAlongSplitting}

First, we want to put log structures on the universal map $f:\sC\to\sW'$ over $\sK_{1,2}'$. This should be compatible with the two relative log stable maps. By the decomposition in Remark \ref{rem:DecompLogCurve}, we have the canonical log structure $\sM^{nd}$ on $\sK_{1,2}'$, which smoothes all the non-distinguished nodes of $\sC$ only. Consider the following log structure on $\sK_{1,2}'$: \[\sM_{\sK_{1,2}'}\cong \sM^{nd}\oplus_{\sO^{*}_{\sK_{1,2}'}}\sM^{\sW'/\sK_{1,2}'}.\] 

Since the map of underlying structure is given, by Remark \ref{rem:LogAtSing} we only need to define maps of log structures $\sM^{\sC/\sK_{1,2}'}_{\sK_{1,2}'}\to\sM_{\sK_{1,2}'}$, where $\sM^{\sC/\sK_{1,2}'}_{\sK_{1,2}'}$ is the canonical log structure as in Section \ref{ss:LogPresCurve}. Denote by $Sing\{\sC/\sK_{1,2}'\}$ the set of connected nodes of $\sC$ over $\sK_{1,2}'$, and $\sN_{p}$ the canonical log structure smoothing $p\in Sing\{\sC/\sK_{1,2}'\}$. By Remark \ref{rem:DecompLogCurve}, \'etale locally we have the following decomposition 
\[\sM^{\sC/\sK_{1,2}'}\cong \sum_{p\in Sing\{\sC/\sK_{1,2}'\}}\sN_{p}.\]
This implies that to define $\sM^{\sC/\sK_{1,2}'}_{\sK_{1,2}'}\to\sM_{\sK_{1,2}'}$, it is enough to locally give $\sN_{p}\to\sM_{\sK_{1,2}'}$. 

Consider any $p\in Sing\{\sC/\sK_{1,2}'\}$ which is not a splitting node. It corresponds to a node in one of the relative log stable maps obtained by splitting from $f$. Since the two relative log stable maps have well-defined log structures, locally we have a well-define map $\sN_{p}\to \sM_{\sK_{1,2}'}$ obtained from the relative log stable map where $p$ sits. Also notice that the splitting node $p_{i}$ for any $i\in M$ only maps to the splitting divisor. Therefore, we only need to define tuples 
\begin{equation}\label{equ:Tuple}
(\sN_{i}\to \sM^{r,spl})_{i\in M}.
\end{equation}
Moreover, we need this map fitting in the following commutative diagram:
\begin{equation}\label{diag:SplitLog}
\xymatrix{
\sM^{1,spl} \ar[d] \ar[r] & \sM^{r,spl}\\
\sN_{i} \ar[ur],
}
\end{equation}
where the map $\sM^{1,spl}\to \sN_{i}$ is given in Lemma \ref{lem:compare-standard-log}. This is because that the map $\sN_{i}\to \sM^{r,spl}$ should compatible with the underlying structure of $f$. 

Conversely, we can construct an arrow of fibered categories $h:\sK_{\eta}\to \sK_{1,2}'$ as follows:
\begin{defn}\label{defn:stack-split-map}
Given an object $\xi_{\eta}\in \sK_{\eta}(T)$ over a scheme $T$, we define an object $\xi_{1,2} \in \sK'_{1,2}(T)$ consisting of the following data:
 \begin{enumerate}
  \item a predeformable map $f_{T}$, which is the underlying map of $\xi_{\eta}$
  \item log relative stable maps $f_{i}$ with discrete data $\Xi_{i}$ for $i=1,2$, which is obtained by splitting $\xi_{\eta}$ along the splitting divisor on target, and the splitting nodes on source curve $C_{T}$ of $\xi_{\eta}$;
  \item a glued log target $\sW'_{T}$, which is the log target of $\xi_{\eta}$, but removing the log structures from the non-distinguished nodes.
 \end{enumerate}
\end{defn}

\begin{lem}\label{lem:CompStack1}
The morphism $h$ is representable.
\end{lem}
\begin{proof}
For any scheme $T$, and any objects $\xi_{1,2}\in \sK'_{1,2}(T)$ and $\xi_{\eta}\in \sK_{\eta}(T)$ such that $h(\xi_{\eta})=\xi_{1,2}$, it is enough to show that a non-trivial automorphism of $\xi_{\eta}$ induces a non-trivial automorphism of $\xi_{1,2}$. Note that $\xi_{\eta}$ gives tuples $(\sN_{i,T}\to \sM^{r,spl}_{T})_{i\in M}$ as in (\ref{equ:Tuple}) satisfying (\ref{diag:SplitLog}). Since outside the splitting nodes, the two objects $\xi_{1,2}$ and $\xi_{\eta}$ determine each other uniquely, it is enough to consider the isomorphisms of $\xi_{\eta}$ given by isomorphisms of the tuples $(\sN_{i,T}\to \sM^{r,spl}_{T})_{i\in M}$. By the commutativity of (\ref{diag:SplitLog}), any such isomorphism is an isomorphism of $\sM^{r,spl}_{T}$, which fix the arrow $\sM^{1,spl}\to\sM^{r,spl}$. But any such non-trivial isomorphisms of $\sM^{r,spl}_{T}$ induce non-trivial automorphisms of the glued log target, hence non-trivial automorphisms of $\xi_{1,2}$. This proves the lemma. 
\end{proof}

Let $\xi_{1,2}\in \sK_{1,2}'(T)$ be an object over a scheme $T$. Denote by $\sK_{\eta}(\xi_{1,2})$ the groupoid consisting of objects $\xi_{\eta}\in \sK_{\eta}(T)$ over $\xi_{1,2}$. Then we have the following result:
\begin{cor}\label{cor:CompGroupoid}
The groupoid $\sK_{\eta}(\xi_{1,2})$ is a set of all tuples $(\sN_{i,T}\to \sM^{r,spl}_{T})_{i\in M}$ satisfying (\ref{diag:SplitLog}).
\end{cor}
\begin{proof}
This follows from Lemma \ref{lem:CompStack1}, and the argument in Section \ref{ss:DefLogAlongSplitting}.
\end{proof}

\subsection{Define log structures on the glued underlying map}
Let $l_{i}=\frac{r}{c_{i}}$, for any $i\in M$. Note that both $\sM^{r,spl}$ and $\sN_{i}$ are rank $1$, locally free. Since they fit in (\ref{diag:SplitLog}), for any $\sN_{i}\to \sM^{r,spl}$, locally we have chart given by the following form:
\[
\xymatrix{
\N \ar[r]^{\times l_{i}} \ar[d] & \N \ar[d] \\
\sN_{i} \ar[r] & \sM^{r,spl}.
}
\]
Note that the map $\sN_{i}\to \sM^{r,spl}$ factors through $\sM^{r,spl}_{l_{i}}\hookrightarrow \sM^{r,spl}$. Thus, to give $\sN_{i}\to \sM^{r,spl}$ satisfying (\ref{diag:SplitLog}), it is equivalent to have an isomorphism $\sN_{i}\to \sM^{r,spl}_{l_{i}}$, which fits in the following commutative diagram:
\[
\xymatrix{
\sM^{1,spl} \ar[d] \ar[r] & \sM^{r,spl}_{l_{i}}\\
\sN_{i} \ar[ur],
}
\]
where the top arrow is induced by $\sM^{1,spl}\to \sM^{r,spl}$.

For each $i\in M$, consider the following catesian diagram:
\[
\xymatrix{
I_{i} \ar[rr] \ar[d]^{q_{i}} && \fT^{c_{i},spl}_{0} \ar[d]\\
\sK'_{1,2} \ar[rr]&& \fT^{c_{i},spl}_{0}\times_{\fT^{1,spl}_{0}}\fT^{c_{i},spl}_{0},
}
\]
where the right column is the diagonal morphism, and the bottom is given by the pair of log structures $\sN_{i}$ and $\sM^{r,spl}_{l_{i}}$. By Section \ref{sss:LogRoot}, the stack $I_{i}$ is a sheaf over $\sK'_{1,2}$ parameterizing isomorphisms between $\sN_{i}$ and $\sM^{r,spl}_{l_{i}}$.

\begin{lem}\label{lem:LogIso}
The morphism $q_{i}:I_{i}\to \sK_{1,2}'$ defined above is \'etale of degree $c_{i}$. It has a section $q'_{i}:\sK_{1,2}'\to I_{i}$, such that $\sK_{1,2}'$ forms a trivial $\sB \mu_{c_{i}}$ gerbe over $I_{i}$. Therefore, $q'_{i}$ is \'etale of degree $\frac{1}{c_{i}}$.  
\end{lem}
\begin{proof}
The first statement follows from $\fT^{c_{i},spl}_{0}\times_{\fT^{1,spl}_{0}}\fT^{c_{i},spl}_{0}\cong \sB \mu_{c_{i}}\times \fT^{c_{i},spl}_{0}$. By pulling back the projection $\sB \mu_{c_{i}}\times \fT^{c_{i},spl}_{0}\to \fT^{c_{i},spl}_{0}$, we obtain the natural section $q_{i}'$ of $q_{i}$, which is \'etale of degree $c_{i}^{-1}$. This proves the second statement. 
\end{proof}

Consider $\sK_{1,2}'':=\prod_{i\in M} I_{i}$, where the product is taking over $\sK_{1,2}'$. We have the following morphism:
\begin{equation}\label{equ:remove-split-log}
q:=\prod_{i\in M}q_{i}: \sK_{1,2}'' \to \sK_{1,2}'.
\end{equation}
By Lemma \ref{lem:LogIso}, the map $q$ is \'etale of degree $\prod_{i\in M}c_{i}$, with a section 
\begin{equation}\label{equ:equip-split-log}
q':=\prod_{i\in M}q'_{i}: \sK_{1,2}'\to \sK_{1,2}'',
\end{equation}
which is \'etale of degree $(\prod_{i\in M}c_{i})^{-1}$.

\begin{rem}\label{rem:LogOverCoarse}
Let us give the moduli interpretation of $\sK_{1,2}''$ relative to $\sK_{1,2}'$. Given a morphism $T\to \sK_{1,2}'$ from a scheme, the stack $\sK_{1,2}''$ associates to $T\to \sK_{1,2}'$ a set of tuples $(\sN_{i,T}\to\sM^{r,spl}_{T})_{i\in M}$ such that:
\begin{enumerate}
 \item For each $i\in M$, the map $\sN_{i,T}\to\sM^{r,spl}_{T}$ fits into (\ref{diag:SplitLog}).
 \item For each $i\in M$, the map $\sN_{i,T}\to\sM^{r,spl}_{T}$ is locally given by the following chart
 \[
 \xymatrix{
 \N \ar[r]^{\times l_{i}} \ar[d] & \N \ar[d] \\
 \sN_{i,T} \ar[r] & \sM^{r,spl}_{T}.
 }
 \]
\end{enumerate}
Therefore, by the argument in Section \ref{ss:DefLogAlongSplitting}, we have a universal family of log stable maps $f_{\sK_{1,2}''}$ over $\sK_{1,2}''$. This family induces a map  $\lambda:\sK_{1,2}''\to \sK_{\eta}$.
\end{rem}

Combining Remark \ref{rem:LogOverCoarse} and Corollary \ref{cor:CompGroupoid}, we obtain:
\begin{prop}
The map $\lambda$ is an isomorphism.
\end{prop}

We thus identify $\sK_{1,2}''$ with $\sK_{\eta}$. Note that the map $q'$ in (\ref{equ:equip-split-log}) is \'etale. Consider $\E_{\sK_{1,2}'/\fT^{r,spl}_{0}}=(q')^{*}\E_{\sK_{\eta}/\fT^{r,spl}_{0}}$. Then $\E_{\sK_{1,2}'/\fT^{r,spl}_{0}}$ gives a perfect obstruction theory of $\sK_{1,2}'$ relative to $(\fT^{r,spl}_{0})^{ext}$. Denote by $[\sK_{1,2}']$ the virtual fundamental class of $\sK_{1,2}'$ defined by $\E_{\sK_{1,2}'/\fT^{r,spl}_{0}}$. Then we have:

\begin{cor}\label{cor:glue-curve}
\[q_{*}[\sK_{\eta}]=\big(\prod_{i\in M}c_{i}\big)[\sK_{1,2}'].\]
\end{cor}
\begin{proof}
This follows from (\ref{equ:remove-split-log}) and Lemma \ref{lem:LogIso}. 
\end{proof}


\subsection{Comparison of virtual fundamental classes of $\sK_{1,2}'$ and $\sK_{1,2}$}
Note that the section $q':\sK_{1,2}'\to\sK_{\eta}$ gives a family of log stable maps $f_{\sK_{1,2}'}$ over $\sK_{1,2}'$, and the perfect obstruction theory $\E_{\sK_{1,2}'/\fT^{r,spl}_{0}}$ can be obtained by applying the construction in Section \ref{ss:ConstructVir} to the log stable map $f_{\sK_{1,2}'}$. The underlying map of $f_{\sK_{1,2}'}$ coincide with the glued predeformable map $f$ as in (\ref{diag:UndGlue}). Denote by $\BL_{\square}$ the complex constructed in Section \ref{ss:ConstructVir} with respect to the log stable map $f_{\sK_{1,2}'}$ on $\sK_{1,2}'$. Let $\BL_{\square_{i}}$ be the corresponding complex with respect to the relative log stable map $f'_{i}$ on $\sK_{1,2}$ with data $\Xi_{i}$, for $i=1,2$.

\begin{prop}\label{prop:vir-glue-tw}
Consider the diagram (\ref{diag:GlueUnderMap}). Then we have $\Delta^{!}([\sK_{1,2}])=[\sK_{1,2}']$.
\end{prop}
\begin{proof}
This proof of this proposition is very similar to \cite[5.8.1]{AF}. By \cite[5.10]{Behrend-Fantechi}, it suffices to produce a diagram of distinguished triangles
\begin{equation}\label{diag:Compdeformation}
\xymatrix{
v^{*}\E_{\sK_{1,2}/\fT^{tw,spl}_{0}} \ar[d] \ar[r] & \E_{\sK_{\eta}/\fT^{tw,spl}_{0}} \ar[d] \ar[r] & \Bev_{\eta}^{*}\BL_{\Delta} \ar[d] \ar[r]^{\ \ \ [1]} & {}\\ 
v^{*}\BL^{log}_{\sK_{1,2} / \fT_{0}^{tw,spl}}  \ar[r] & \BL^{log}_{\sK_{\eta}/\fT^{tw,spl}_{0}}  \ar[r] & \BL_{\sK_{\eta}/\sK_{1,2}}  \ar[r]^{ \ \ \ \ [1]} & {}\\ 
}
\end{equation}
Since the map $\Delta$ in (\ref{diag:GlueUnderMap}) is a regular embedding, we have $\BL_{\Delta}\cong N_{\Delta}^{\vee}[1]$. 

Consider (\ref{diag:UndGlue}). Denote by $v:\sC_{1}'\coprod\sC_{2}'\to\sC$ the normalization map, and $\iota:\sG\to \sC$ the embedding. We have the following standard normalization triangle
\begin{equation}\label{diag:OnFamily}
\BL_{\square}^{\vee} \to \upsilon_{*}L\upsilon^{*}\BL_{\square}^{\vee} \to \iota_{*}L\iota^{*}\BL_{\square}^{\vee} \stackrel{[1]}{\rightarrow}
\end{equation}
and a natural decomposition
\[\upsilon_{*}L\upsilon^{*}\BL_{\square}^{\vee} = \upsilon_{1*}L\upsilon^{*}_{1}\BL_{\square}^{\vee} \oplus \upsilon_{2*}L\upsilon^{*}_{2}\BL_{\square}^{\vee}.\]
\end{proof}

\begin{lem}
\[L\upsilon^{*}_{1}\BL_{\square}^{\vee}=\BL_{\square_{1}}^{\vee}, \mbox{\ \ \ } L\upsilon^{*}_{2}\BL_{\square}^{\vee} = \BL_{\square_{2}}^{\vee}\]
and 
\[L\iota^{*}\BL_{\square}^{\vee}=(f\circ \iota)^{*} T_{D}[-1].\]
\end{lem}
\begin{pfoflem}
Note that the underlying map $f$ is obtained by gluing the two underlying maps $\underline{f}_{1}'$ and $\underline{f}_{2}'$ along the marked point as in (\ref{diag:UndGlue}). Away from the splitting nodes, the log structure on $f_{\sK_{1,2}'}$ is obtained by pulling-back that from $f_{i}'$, locally we have a canonical isomorphism $L\upsilon^{*}_{i}\BL_{\square}^{\vee}\cong\BL_{\square_{i}}^{\vee}$. 

Near the splitting nodes, we can directly check that the log differentials of the log maps $f_{\sK_{1,2}'}$ and $f_{i}'$ are identical to the sheaf of differentials of transversal maps as in \cite[A.2]{AF}. Therefore, the same argument as in \cite[5.8.2]{AF} shows that locally near the splitting nodes, there are canonical isomorphisms $L\upsilon^{*}_{i}\BL_{\square}^{\vee}\cong\BL_{\square_{i}}^{\vee}$. We check that those local isomorphisms can be glued together to give a global one $L\upsilon^{*}_{i}\BL_{\square}^{\vee}=\BL_{\square_{i}}^{\vee}$. 

For the same reason, we have $L\iota^{*}\BL_{\square}^{\vee}=(f\circ \iota)^{*} T_{D}[-1]$. This finishes the proof of the lemma.
\end{pfoflem}

By (\ref{diag:OnFamily}) and the above lemma, we have:
\begin{equation}\label{diag:OnFamily2}
\BL_{\square}^{\vee} \to \upsilon_{1*}\BL_{\square_{1}}^{\vee} \oplus \upsilon_{1*}\BL_{\square_{2}}^{\vee} \to \iota_{*}(f\circ \iota)^{*} T_{D} \stackrel{[1]}{\rightarrow}
\end{equation}
Denote by $p:\sC \to \sK_{1,2}'$ and $p_{i}:\sC'_{i}\to \sK_{1,2}'$ the canonical projections. Note that we have
\[p_{*}\iota_{*}(f\circ \iota)^{*} T_{D} = (\Bev'_{\eta})^{*}N_{\Delta}.\]
By applying $Rp_{*}$, dualizing and rotate (\ref{diag:OnFamily2}) we have
\[(Rp_{1*}\BL_{\square_{1}}^{\vee})^{\vee}[-1]\oplus(Rp_{2*}\BL_{\square_{2}}^{\vee})^{\vee}[-1]\to (Rp_{*}\BL_{\square}^{\vee})^{\vee}[-1]\to \Bev_{\eta}^{*'}N_{\Delta}[1]\stackrel{[1]}{\rightarrow}.\] 
This triangle fits into (\ref{diag:Compdeformation}) as required.  \\

Combining Corollary \ref{cor:glue-curve} and \ref{prop:vir-glue-tw}, we have the following:
\begin{cor}
\[q'_{*}[\sK_{\eta}]=\big(\prod_{i\in M}c_{i}\big)\Delta^{!}([\sK_{1,2}]).\]
\end{cor}

Denote by $[\Delta]$ the class $\Delta_{*}(D^{|M|})$. By \cite[6.3.4]{Fulton}, we have 
\[q'_{*}[\sK_{\eta}] = (\prod_{i\in M} c_{i})\cdot \Bev^{*}_{1,2}[\Delta]\cap [\sK_{1,2}].\]
Since 
\[[\Delta]=\prod_{i\in M}(\sum_{\delta_{i}\in F}\delta_{i}\times \delta_{i}^{\vee}),\] 
we obtain
\[q'_{*}[\sK_{\eta}] = \prod_{i\in M} \big(c_{i}\sum_{\delta_{i}\in F} \Bev^{*}_{\Xi_{1},i}\delta_{i}\times\Bev^{*}_{\Xi_{2},i}\delta_{i}^{\vee}\big)\cap [\sK_{1,2}].\]
By Lemma \ref{lem:GlueTarget} we have

\begin{cor}\label{cor:glue-vir}
\[(u_{\eta}\circ q')_{*}[\sK_{\eta}] = r(\eta)^{-1}\cdot\prod_{i\in M} \big(c_{i}\sum_{\delta_{i}\in F} \Bev^{*}_{\Xi_{1},i}\delta_{i}\times\Bev^{*}_{\Xi_{2},i}\delta_{i}^{\vee}\big)\cap [\sK_{\Xi_{1}}\times\sK_{\Xi_{2}}].\]
\end{cor}

\subsection{Proof of Theorem \ref{thm:main}}
Same as the situation in \cite[5.10]{AF}, the stack $\sK_{\eta}$ carries two universal families of contracted curves: 
\begin{enumerate}
 \item the disconnected family $\sC'\to \sK_{\eta}$ with sections $s_{i}'$, and evaluations $\Bev_{i}'$, which are the pull-back of the contracted family of $\sC_{1}\coprod \sC_{2}\to \sK_{\Xi_{1}}\times\sK_{\Xi_{2}}$;
 \item the connected family $C_{\eta}\to \sK_{\eta}$ as in (\ref{diag:contract-family}) with sections $s_{i}$, and evaluations $\Bev_{i}$.
\end{enumerate}
Note that the pull-back of the class $\psi_{i}$ of the sheaf $s_{i}^{*}(\omega_{\C_{\eta}/\sK_{\eta}})$ coincides with the class $(s_{i}')^{*}(\omega_{\sC_{1}\coprod \sC_{2}/ \sK_{\Xi_{1}}\times\sK_{\Xi_{2}}})$, and the same for the pull-back of $\gamma_{i}$ via the evaluation maps. Now we have:

\begin{align*}
\left\langle \prod_{i=1}^{n}{\tau_{m_{i}}(\gamma_{i})}\right\rangle_{g,n,\beta}^{W_{0}} &= \sum_{\eta\in \Omega}\frac{r(\eta)}{|M|!}\mbox{deg}\big((s\circ t_{\eta})^{*}(\prod_{i\in N}\psi_{i}^{m_{i}}\cdot\Bev_{i}^{*}\gamma_{i})\cap[\sK_{\eta}]\big) \\
   & (\mbox{by the projection formula and Proposition \ref{prop:first-decomp-vir}})\\
&= \sum_{\eta\in \Omega}\frac{r(\eta)}{|M|!}\mbox{deg}\big((u_{\eta}\circ h)^{*}(\prod_{i\in N}\psi_{i}^{'m_{i}}\cdot\Bev_{i}^{'*}\gamma_{i})\cap[\sK_{\eta}]\big)\\
   & (\mbox{by the above discussion}) \\
&= \sum_{\eta\in \Omega}\frac{r(\eta)}{|M|!}\frac{\prod_{j\in M}c_{j}}{r(\eta)}\mbox{deg}\Big((\prod_{i\in N}\psi_{i}^{'m_{i}}\cdot\Bev_{i}^{'*}\gamma_{i})\\
 & \ \ \ \ \ \ \ \cdot \prod_{j\in M}\big(\sum_{\delta_{j}\in F}  \Bev^{*}_{\Xi_{1},j}\ \delta_{j}\times\Bev^{*}_{\Xi_{2},j}\ \delta_{j}^{\vee}\big)\cup[\sK_{\Xi_{1}}]\times[\sK_{\Xi_{2}}]\Big)\\
    & (\mbox{by the projection formula and Corollary \ref{cor:glue-vir}}) \\
&= \sum_{\eta\in \Omega}\frac{\prod_{j\in M}c_{j}}{|M|!}\sum_{\delta_{j}\in F}(-1)^{\epsilon}\left\langle\prod_{i\in N_{1}}\tau_{m_{i}}(\gamma_{i})|\prod_{j\in M}\delta_{j}\right\rangle_{\Xi_{1}}^{X_{1},D}\\
& \cdot \left\langle\prod_{i\in N_{2}}\tau_{m_{i}}(\gamma_{i})|\prod_{j\in M}\delta_{j}^{\vee}\right\rangle_{\Xi_{2}}^{X_{2},D}
\end{align*}

This finishes the proof of the theorem. $\square$

\appendix

\section{Prerequisites On Logarithmic Geometry}

\subsection{Basic definitions and properties}
Following \cite{KKato} and \cite{Ogus}, we first recall some basic terminologies on logarithmic geometry.

\subsubsection{Monoids} 
A {\em monoid} is a commutative semi-group with a unit. We usually use $``+"$ and $``0"$ to denote the binary operation and the unit of a monoid respectively. A morphism between two monoids is required to preserve the unit. 

Let $P$ be a monoid. We can associate a group
\[P^{gp}:=\{(a,b) |(a,b)\sim(c,d) \ \mbox{if} \ \exists s \in P \ \mbox{such that} \ s+a+d=s+b+c\}.\]
The monoid $P$ is called {\em integral} if the natural map $P\to P^{gp}$ is injective. It is called {\em saturated} if it is integral, and satisfies that for any $p\in P^{gp}$ if $n\cdot p\in P$ for some positive integer $n$ then $p\in P$. A monoid $P$ is called {\em fine} if it is integral and finitely generated.

A morphism $h: Q\to P$ between integral monoids is called {\em integral} if for any $a_{1}, a_{2}\in Q$, and $b_{1},b_{2}\in P$ which satisfy $h(a_{1})b_{1}=h(a_{2})b_{2}$, there exist $a_{2}, a_{4}\in Q$ and $b\in P$ such that $b_{1}=h(a_{3})b$ and $a_{1}a_{3}=a_{2}a_{4}$. This is equivalent to say that the morphism of monoid algebras $\Z[Q]\to Z[P]$ is flat.

A monoid $P$ is called {\em sharp} if there are no other unit except $0$. A nonzero element $p$ in a sharp monoid $P$ is called {\em irreducible} if $p=a+b$ implies either $a=0$ or $b=0$. Denote by $Irr(P)$ the set of irreducible elements in a sharp monoid $P$. A fine monoid $P$ is called {\em free} if $P\cong \N^{n}$ for some positive integer $n$.

\subsubsection{Logarithmic structures}\label{ss:DefLogStr}
Let $X$ be a scheme. A {\em pre-log} structure on $X$ is a pair $(\sM,\exp)$, which consists of a sheaf of monoids $\sM$ on the \'etale site $X_{\acute{e}t}$ of $X$, and a morphism of sheaves of monoids $\exp: \sM\to \sO_{X}$, called the structure morphism of $\sM$. Here we view $\sO_{X}$ as a monoid under multiplication. 

A pre-log structure $\sM$ on $X$ is called a {\em log structure} if $\exp^{-1}(\sO_{X}^{*})\cong \sO^{*}_{X}$ via $\exp$. We sometimes omit the morphism $\exp$, and use $\sM$ to denote the log structure if no confusion could arise. We call the pair $(X,\sM)$ a {\em log scheme}. 

Given two log structures $\sM$ and $\sN$ on $X$, a {\em morphism of the log structures} $h:\sM\to\sN$ is a morphism of sheaves of monoids which are compatible with the structure morphisms of $\sM$ and $\sN$.

Given a pre-log strucutre $\sM$ on $X$, we obtain a log structure $\sM^{a}$ given by 
\[\sM^{a}:=\sM\oplus_{\exp^{-1}(\sO^{*}_{X})}\sO_{X}^{*}.\]
Such $\sM^{a}$ is called the {\em associated log structure of $\sM$}. Consider a morphism of schemes $f:X\to Y$, and a log structure $\sM_{Y}$ on $Y$. We can define the {\em pull-back log structure} $f^{*}(\sM_{Y})$ to be the log structure associated to the pre-log structure 
\[f^{-1}(\sM_{Y})\to f^{-1}(\sO_{Y})\to \sO_{X}.\]

Consider two log schemes $(X,\sM_{X})$ and $(Y,\sM_{Y})$. A {\em morphism of log schemes} $(X,\sM_{X})\to (Y,\sM_{Y})$ is a pair $(f,f^{\flat})$, where $f:X\to Y$ is a morphism of the underlying schemes, and $f^{\flat}:f^{*}(\sM_{Y})\to \sM_{X}$ is a morphism of log structures on $X$. The morphism $(f,f^{\flat})$ is called {\em strict} if $f^{\flat}$ is an isomorphism of log structures. It is called {\em vertical} if the quitient sheaf of monoids $\sM_{X}/f^{*}(\sM_{Y})$ is a sheaf of groups under the induced monoidal operation.

\subsubsection{Charts of log structures}\label{ss:ChartLog}
Let $(X,\sM)$ be a log scheme, and $P$ a fine monoid. Denote by $P_{X}$ the constant sheaf of monoid $P$ on X. A {\em chart} of $\sM$ is a morphism $P_{X}\to \sM$ such that the associated log structure of the composition $P_{X}\to \sM\to\sO_{X}$ is $\sM$. The log structure $\sM$ is called a {\em fine log structure} on $X$ if a chart exists \'etale locally everywhere on $X$. If the charts are all given by fine and saturated monoids then $\sM$ is called an {\em fs} log structure. In this paper, we only consider fs log structures.  

Let $\CharM=\sM/\sO_{X}^{*}$ be the quotient sheaf. We call it the {\em characteristic} of the log structure $\sM$. It is useful to notice that $\overline{f^{*}(\sM)}=f^{-1}(\CharM)$ for any morphism of schemes $f: Y\to X$. For any closed point $x\in X$, denote by $\bar{x}$ the separable closure of $x$. A fine log structure $\sM$ is called locally free if for any $x\in X$, we have $\CharM_{\bar{x}}\cong \N^{n}$ for some positive integer $r$. 

Consider a morphism of log structures $h:\sM\to\sN$ over $X$. The morphism $h$ is called {\em simple} at $p\in X$, if $\CharM_{\bar{p}}\to\overline{N}_{\bar{p}}$ is injective, and for any $e\in Irr(\CharM_{\bar{p}})$ there exists an element $e'\in Irr(\overline{N}_{\bar{p}})$, such that $\bar{h}(e)=l\cdot e'$ for some positive integer $l$. Here $\bar{h}$ is the map of monoids induced by $h$.

Let $\CharM_{\bar{x}}^{gp,tor}$ be the torsion part of $\CharM_{\bar{x}}^{gp}$. The following result is very useful for creating charts.

\begin{prop}\cite[2.1]{LogStack}
Using the notation as above, there exist an fppf neighborhood $f:X'\to X$ of $x$, and a chart $\beta:P\to f^{*}(\sM)$ such that for some geometric point $\bar{x}'\to X'$ lying over $x$, the natural map $P\to f^{-1}\CharM_{\bar{x}'}$ is bijective. If $\CharM_{\bar{x}}^{gp,tor}\otimes k(x)=0$, then such a chart exists in an \'etale neighborhood of $x$.
\end{prop}

\begin{rem}
In this paper, we only work with fs log structures over field of characteristic $0$. The above proposition implies that in such situation, there exists a section of $\sM_{\bar{x}}\to\CharM_{\bar{x}}$, which gives a chart \'etale locally near $x$.
\end{rem}

Consider a morphism $f:(X,\sM_{X})\to(Y,\sM_{Y})$ of fine log schemes. A {\em chart} of $f$ is a triple $(P_{X}\to\sM_{X},Q_{Y}\to\sM_{Y},Q\to P)$, where $P_{X}\to \sM_{X}$ and $Q_{Y}\to \sM_{Y}$ are charts of $\sM_{X}$ and $\sM_{Y}$ respectively, and $Q\to P$ is a morphism of monoids such that the following diagram commutes:
\[\xymatrix{
Q_{X} \ar[r] \ar[d] & P_{X} \ar[d] \\
f^{*}(\sM_{Y}) \ar[r] & \sM_{X}.
}
\]
Charts of morphisms of fine log schemes exist \'etale locally by the following result:
\begin{prop}\cite[2.2]{LogStack}
Notations as above, suppose that $Q_{Y}\to\sM_{Y}$ is a chart. Then \'etale locally on $X$, there exist a chart $P_{X}\to\sM_{X}$ and an injective morphism of monoids $Q\to P$, such that the triple $(P_{X}\to\sM_{X},Q_{Y}\to\sM_{Y},Q\to P)$ gives a chart for $f$ \'etale locally on $X$. If $f$ is a morphism of fs log schemes and if $Q$ is saturated and torsion free, then we can choose $P$ to be also saturated and torsion free in the chart of $f$.
\end{prop}

\begin{rem}
Consider a morphism of log schemes $f:(X,\sM_{X})\to(Y,\sM_{Y})$. With the help of charts, we can describe the log smoothness properties of $f$ that we will use later. The log map $f$ is called log smooth if \'etale locally, there is a chart $(P_{X}\to\sM_{X},Q_{Y}\to\sM_{Y},Q\to P)$ of $f$ such that:
\begin{enumerate}
 \item $\Ker{Q^{gp}\to P^{gp}}$ and the torsion part of $\Coker(Q^{gp}\to P^{gp})$ are finite groups;
 \item the induced map $X\to Y\times_{Spec(\Z[Q])}Spec\Z[p]$ is smooth in the usual sense.
\end{enumerate}

The map $f$ is called {\em integral} if for every $p\in X$, the induced map $\CharM_{f(\bar{p})}\to\CharM_{\bar{p}}$ is integral. In general, the underlying structure map of a log smooth morphism need not be flat. However, it was shown in \cite[4.5]{KKato} that the underlying map of a log smooth and integral morphism is flat.
\end{rem} 

\subsection{Olsson's Log Stacks}\label{s:LogStack}
We follow \cite{LogStack} to introduce the algebraic stack parameterizing log structures. Consider a base scheme $S$, and an algebraic stack $\sX$ over $S$ in the sense of \cite{Artin}. This means that the diagonal $\sX\to\sX\times_{S}\sX$ is representable and of finite type, and there exists a surjective smooth morphism $X\to \sX$ from a scheme. Now we can define a fine log structure $\sM_{\sX}$ on $\sX$ by repeating the definitions in \ref{ss:DefLogStr} and \ref{ss:ChartLog} but using the lisse-\'etale site instead of the \'etale site. See \cite[Section 5]{LogStack} for details.

For any $S$-scheme $T$, and an arrow $g:T\to \sX$, we obtain a fine log structure $g^{*}(\sM_{\sX})$ on the lisse-\'etale site $T_{lis\mbox{-}et}$ of $T$. It is shown in \cite[5.3]{LogStack} that such $g^{*}(\sM_{\sX})$ is isomorphic to a unique fine log structure on the \'etale site $T_{\acute{e}t}$ of $T$. Thus, we can still use $g^{*}(\sM_{\sX})$ to denote this new log structure on $T$. By pull-back the log structure $\sM_{\sX}$, we define a functor from schemes over $\sX$ to the category of fine log schemes over $S$. The stack $\sX$ associated with this functor is called a {\em log stacks} in \cite{FKato}. A fine log scheme $(X,\sM_{X})$ can be naturally viewed as a log algebraic stack.

Consider the log algebraic stack $(\sX,\sM_{X})$. We define a fibered category $\sL og_{(\sX,\sM_{\sX})}$ over $\sX$. Its objects are pairs $(g:X\to\sX,g^{*}(\sM_{\sX})\to \sM_{X})$, where $g$ is a map from scheme $X$ to $\sX$, and $g^{*}(\sM_{\sX})\to \sM_{X}$ is a morphism of fine log structures on $X$. An arrow $(g:X\to\sX,g^{*}(\sM_{\sX})\to \sM_{X})\to(h:Y\to\sX,h^{*}(\sM_{\sX})\to \sM_{Y})$ is a strict morphism of log schemes $(X,\sM_{X})\to (Y,\sM_{Y})$, such that the underlying map $X\to Y$ is a morphism over $\sX$, and we have the following cartesian diagram:
\[
\xymatrix{
(X,\sM_{X}) \ar[r] \ar[d] & (Y,\sM_{Y}) \ar[d] \\
(X,g^{*}(\sM_{\sX})) \ar[r] & (Y,h^{*}(\sM_{\sX})).
}
\]

\begin{thm}\cite[5.9]{LogStack}
The fibered category $\sL og_{(\sX,\sM_{\sX})}$ is an algebraic stack locally of finite presentation over $\sX$.
\end{thm}

\subsection{Root of log structures from smooth divisors}\label{ss:RootLog}
Here, we collect some results of locally free log structures and their roots as we will used in our construction. We refer to \cite{BV} for more details of the root construction.

\subsubsection{Log structures associated to normal crossing divisors.}\label{sss:DivLog}
This is an important example given in \cite[1.5]{KKato}. Let $X$ be a smooth scheme, and $D$ is reduced divisor in $X$ with normal crossings. We define a fine log structure on $X$:
\[\sM^{D}:=\{g\in \sO_{X} | g \mbox{ is invertible outside }D\}\subset \sO_{X}.\]
For each point $p\in D$, let $\{g_{i}\}_{i=1}^{n}$ be the set of local coordinates near $p$, such that $D$ is given by the vanishing of $g_{1}\cdots g_{n}$. Then the log structure $\sM^{D}_{p}$ is generated by $\{\log g_{i}\}_{i=1}^{n}$, where $\log g_{i}$ is the pre-image of $g_{i}$ in the log structure. Thus \'etale locally near $p$ we have a chart 
\[\N^{n}\to \sM^{D} \ \ \ e_{i}\mapsto g_{i},\]
where $e_{i}$ is the standard generators of $\N^{n}$. The above chart gives an isomorphism $\N^{n}\cong \CharM_{p}^{D}$. Thus, the log structure $\sM^{D}$ is locally free, and its rank at a point $p$ equals the number of components of $D$ at $p$.

Consider the case when $D$ is a finite union of reduced divisors in $X$ with normal crossings. Denote by $D=\coprod_{j}{D_{j}}$. The same definition as above gives a log structure $\sM^{D}$ on $X$. Similarly, for each $D_{j}$ we associate a log structure $\sM^{D_{j}}$ on $X$ as above. We have the following decomposition:
\[\sM^{D}\cong \sum_{j}{\sM^{D_{j}}},\]
where the sum is taking over $\sO^{*}_{X}$.

We can assume that $X$ is an algebraic stack, and $D$ is a finite union of reduced divisors in $X$ with normal crossings. Then we can still define the locally free log structure $\sM^{D}$ on $X$ using smooth topology. The decomposition of $\sM^{D}$ into amalgamated sum of $\sM^{D_{j}}$ still holds in this case.

\subsubsection{Root stacks of divisors}
For later use, we introduce the theory of root stacks developed in \cite{Cadman} and \cite[Appendix B]{AGV}.

Let $X$ be an algebraic stack, and $D$ an effective divisor on $X$. Such data corresponds to the line bundle $\sO_{X}(D)$ with its canonical section $\textbf{1}_{D}$, therefore induces a map $X\to \sA=[\A^{1}/\G_{m}]$. Consider the degree $r$ map $\nu_{r}: \sA\to\sA$, given by $t\mapsto t^{r}$ where $t$ is the coordinates of $\A^{1}$. We form a stack $X_{D,r}=X\times_{D,\sA,\nu_{r}}\sA$. Given a $X$-scheme $S$, the objects in $X_{D,r}(S)$ are tuples $(f,M,\phi,s)$ where
\begin{enumerate}
 \item $f:S\to X$ is a morphism;
 \item $M$ is a line bundle on $S$;
 \item $\phi:M^{r}\to f^{*}\sO_{X}(D)$ is an isomorphism;
 \item $s\in H^{0}(M)$ such that $\phi(s^{r})=f^{*}(\textbf{1}_{D})$.
\end{enumerate}
We call $X_{D,r}$ the $r$-th root stack of $D$. It is of Deligne-Mumford type over $X$. 

Consider the case where $X$ is smooth, and the effective divisor $D$ is smooth in $X$.  We would like to consider the local structure of $X_{D,r}$. Thus, we can assume that $X=\Spec k[x_{1},\cdots,x_{l}]$, and $D$ is given by the vanishing of $x_{1}$. Then in this case, the stack $X_{D,r}$ is given by the following stack quotient:
\[[\Spec(k[y,x_{2},\cdots,x_{l}]/(y^{r}=x_{1}))/\mu_{r}]\]
where $\mu_{r}$ is a finite cyclic group of order $r$ invertible in $k$, and for any $u\in \mu_{r}$, the action is given by $u:y\mapsto u\cdot y$, and fix all other coordinates. 

The above local description shows that when both $X$ and $D$ are smooth, the stack $X_{D,r}$ is also smooth. Denote by $\sD\subset X_{D,r}$ the smooth divisor given by the vanishing of the local coordinate $y$. The inverse image of $D$ in $X_{D,r}$ is $r\cdot\sD$. In fact, the morphism $\sD\to D$ makes $\sD$ an $\mu_{r}$-gerbe over $D$. We call the $\sD$ is the $r$-th root of $D$ in $X_{D,r}$.

Consider $\DD=(D_{1},\cdots,D_{n})$ an $n$-tuple of effective divisors $D_{i}\subset X$, and consider $\Br=(r_{1},\cdots,r_{n})$ an $n$-tuple of positive integers. We use the following notation:
\[X_{\DD,\Br}=X_{D_{1},r_{1}}\times_{X}\cdots\times_{X}X_{D_{n},r_{n}}.\]

\subsubsection{Root of log structures}\label{sss:LogRoot}
Let $X$ be a smooth algebraic stack, and $D=\coprod_{j}{D_{j}}$ be a normal crossings divisor given by union of smooth divisors $D_{j}$ on $X$. We assume that $D_{j}\neq D_{i}$ if $j\neq i$. Let $\Br=(r_{1},\cdots,r_{n})$ be an $n$-tuple of positive integers. By the argument in \ref{sss:DivLog}, we have locally free log structures $\sM^{D}$ and $\sM^{D_{j}}$ on $X$ for all $j$ such that 
\[\sM^{D}\cong \sM^{D_{1}}\oplus_{\sO_{X}^{*}}\cdots\oplus_{\sO_{X}^{*}}\sM^{D_{n}}.\]
Let $\N\to \sM^{D_{j}}$ be a local chart for $\sM^{D_{j}}$, then the above decomposition gives a local chart $\N^{n}\to\sM^{D}$. 

Now we define a fibered category $\sX$ over $X$. Given a morphism $f:Y\to X$ from a scheme $Y$, the fiber $\sX(Y)$ is the groupoid of simple morphisms of log structures $f^{*}(\sM^{D})\to\sM$, such that for each geometric point $\bar{p}\in Y$, we locally have a chart near $\bar{p}$
\[
\xymatrix{
\N^{n} \ar[r]^{\oplus_{j}\times r_{j}} \ar[d] & \N^{n} \ar[d]\\
f^{*}(\sM^{D}) \ar[r]& \sM,
}
\]
where each integer $r_{j}$ corresponds to the factor of $\N$ in $\N^{n}$ given by the divisor $D_{j}$. A morphism between two objects
\[(f^{*}(\sM^{D})\to\sM_{1})\to (f^{*}(\sM^{D})\to\sM_{2})\]
is an isomorphism of log structures $\sM_{1}\to\sM_{2}$ fitting in the following commutative diagram:
\[
\xymatrix{
  &f^{*}(\sM_{X}) \ar[rd] \ar[dl] & \\
\sM_{1} \ar[rr] && \sM_{2}.
}
\] 

Denote by $\DD=(D_{1},\cdots,D_{n})$. It was shown in \cite[3.6]{Cadman} and \cite[Section 4]{MO} that $\sX$ is the stack $X_{\DD,\Br}$ defined above.



\end{document}